\documentclass[a4paper,11pt]{amsart}
\usepackage[left=2.7cm,right=2.7cm,top=3.5cm,bottom=3cm]{geometry}

\usepackage{amssymb,latexsym,amsmath,amsthm,amscd}
\usepackage{stmaryrd}

\usepackage{graphicx}
\usepackage{dsfont}
\usepackage{longtable,tabu}
\usepackage[all]{xy}
\usepackage{mathrsfs}
\usepackage{color}
\definecolor{Blue}{rgb}{0.3,0.3,0.9}
\usepackage{array}
\usepackage{tabu}
\usepackage{hyperref}

\usepackage[T2A,OT1]{fontenc}
\DeclareSymbolFont{cyrillic}{T2A}{cmr}{m}{n}
\DeclareMathSymbol{\Sha}{\mathalpha}{cyrillic}{216}

\newtheorem{thm}{Theorem}[section]
\newtheorem{pro-thm}[thm]{Main Theorem}
\newtheorem{def-thm}[thm]{Definition-Theorem}
\newtheorem{cor}[thm]{Corollary}
\newtheorem{lem}[thm]{Lemma}

\newtheorem{prop}[thm]{Proposition}

\newtheorem{conj}[thm]{Conjecture}

\newtheorem*{conjA'}{Conjecture~A$'$}

\newtheorem*{assumption-spade}{Hypothesis~$\spadesuit$}
\theoremstyle{definition}
\newtheorem{defn}[thm]{Definition}
\newtheorem{hyp}[thm]{Hypotheses}
\theoremstyle{remark}
\newtheorem{rem}[thm]{Remark}
\newtheorem{intro-rem}{Remark}
\numberwithin{equation}{section}

\newtheorem{claim}[thm]{Claim}

\newcommand{\Sel}{{\mathfrak{Sel}}}

\newcommand{\pp}{\mathfrak{p}}
\newcommand{\ppbar}{\overline{\mathfrak{p}}}
\newcommand{\qq}{\mathfrak{q}}

\newcommand{\bQ}{\mathbf{Q}}
\newcommand{\bZ}{\mathbf{Z}}

\def\cO{\mathcal O}

\newcommand{\Z}{\mathbf{Z}}

\def\Sar{{S}}
\def\fSar{{\mathfrak{S}}}

\def\makeop#1{\expandafter\def\csname#1\endcsname
	{\mathop{\rm #1}\nolimits}\ignorespaces}
\makeop{Hom}   \makeop{End}   \makeop{Aut}   
\makeop{Pic} \makeop{Gal}       \makeop{Div} \makeop{Lie}
\makeop{PGL}   \makeop{Corr} \makeop{PSL} \makeop{sgn} \makeop{Spf}
\makeop{Tr} \makeop{Nr} \makeop{Fr} \makeop{disc}
\makeop{Proj} \makeop{supp} \makeop{ker}   \makeop{Im} \makeop{dom}
\makeop{coker} \makeop{Stab} \makeop{SO} \makeop{SL} \makeop{SL}
\makeop{Cl}    \makeop{cond} \makeop{Br} \makeop{inv} \makeop{rank}
\makeop{id}    \makeop{Fil} \makeop{Frac}  \makeop{GL} \makeop{SU}
\makeop{Trd}   \makeop{Sp} \makeop{Tr}    \makeop{Trd} \makeop{Res}
\makeop{ind} \makeop{depth} \makeop{Tr} \makeop{st} \makeop{Ad}
\makeop{Int} \makeop{tr}    \makeop{Sym} \makeop{can} \makeop{SO}
\makeop{torsion} \makeop{GSp} \makeop{Tor}\makeop{Ker} \makeop{rec}
\makeop{Ind} \makeop{Coker}
\makeop{vol} \makeop{Ext} \makeop{gr} \makeop{ad}
\makeop{Gr}\makeop{corank} \makeop{Ann}
\makeop{Hol} 
\makeop{Fitt} \makeop{Mp} \makeop{CAP}

\def\Sel{\mathfrak{Sel}}

\newcommand{\beqcd}[1]{\begin{equation*}\label{#1}\tag{#1}}
\newcommand{\eeqcd}{\end{equation*}}

\newcommand{\fS}{{\mathfrak{S}}}

\begin{document}
	
\title
{On anticyclotomic variants of the $p$-adic Birch and Swinnerton-Dyer conjecture}
\author[B.~Agboola]{Adebisi Agboola}
\email{agboola@ucsb.edu}
\address[]{Department of Mathematics, University of California Santa Barbara, CA 93106, USA}

\author[F.~Castella]{Francesc Castella}
\email{castella@ucsb.edu}
\address[]{Department of Mathematics, University of California Santa Barbara, CA 93106, USA}
	
\thanks{During the preparation of this paper, Castella was partially supported by NSF grant DMS-1801385, 1946136.}
\date{\today}




\begin{abstract}
We formulate analogues of the Birch and Swinnerton-Dyer conjecture for the $p$-adic $L$-functions of Bertolini--Darmon--Prasanna 
attached to elliptic curves $E/\bQ$ at primes $p$ of good ordinary reduction. 
%
Using Iwasawa theory, we then prove under mild hypotheses 
one of the  inequalities predicted by the rank part of our conjectures, as well as the predicted leading coefficient formula up to a $p$-adic unit. 

Our conjectures are very closely related to conjectures of Birch and Swinnerton-Dyer type formulated by Bertolini--Darmon in 1996 for certain Heegner distributions, and as application of our results we  also obtain the proof of an inequality in the rank part of their conjectures.
\end{abstract}


\maketitle
\setcounter{tocdepth}{1}
\tableofcontents

\section{Introduction}



Let $E/\bQ$ be an elliptic curve of conductor $N$, let $p>2$ be a  prime of good ordinary reduction for $E$, and let $K$ be an imaginary quadratic field of discriminant prime to $Np$. Let $K_\infty/K$ be the anticyclotomic $\bZ_p$-extension of $K$, 
and set $\Gamma_\infty={\rm Gal}(K_\infty/K)$ and $\Lambda=\bZ_p[[\Gamma_\infty]]$.

Assume that $K$ satisfies the \emph{Heegner hypothesis} relative to $N$, i.e., that
\begin{equation}\label{eq:Heeg}
\textrm{every prime factor of $N$ splits in $K$.}\tag{Heeg}
\end{equation}
This condition implies that the root number of $E/K$ is $-1$. 
%
%
Assume in addition that 
\begin{equation}\label{eq:spl}
\textrm{$p=\mathfrak{p}\overline{\pp}$ splits in $K$,}\tag{spl}
\end{equation} 
and let $\hat{\mathcal{O}}$ be the completion of the ring of integers of the maximal unramified extension of $\bQ_p$. Let $f\in S_2(\Gamma_0(N))$ be the newform associated with $E$. In \cite{bdp1} (as later strengthened in \cite{braIMRN, cas-hsieh1}), 
Bertolini--Darmon--Prasanna introduced a $p$-adic $L$-function
\[
\mathscr{L}_\pp(f)\in
\Lambda_{\hat{\cO}}:=\Lambda\hat{\otimes}_{\bZ_p}\hat{\mathcal{O}}
\]
with $L_\pp(f):=\mathscr{L}_\pp(f)^2$ interpolating the central critical values for the twists of $f/K$ by certain infinite order characters of $\Gamma_\infty$. In this paper we formulate and study analogues of the Birch and Swinnerton-Dyer conjecture for these $p$-adic $L$-functions.

Any continuous character $\chi:\Gamma_\infty\rightarrow\hat{\cO}^\times$ extends to a map $\chi:\Lambda_{\hat{\cO}}\rightarrow\hat{\cO}$, and  we  write $L_\pp(f)(\chi)$ for $\chi(L_\pp(f))$. The trivial character $\mathds{1}$ of $\Gamma_\infty$ lies outside the range of $p$-adic interpolation for $L_\pp(f)$, and one of the main results of \cite{bdp1} is a formula of $p$-adic Gross--Zagier type for this value:
\begin{equation}\label{eq:bdp}
L_\pp(f)(\mathds{1})=u_K^{-2}c_E^{-2}\cdot\biggl(\frac{1-a_p(E)+p}{p}\biggr)^2\cdot\log_{\omega_E}(z_K)^2
\end{equation}
(see \cite[Thm.~5.13]{bdp1}, as specialized in \cite[Thm.~3.12]{bdp3} to the case where $f$ has weight $k=2$). Here $a_p(E):=p+1-\#E(\mathbf{F}_p)$ and $u_K:=\frac{1}{2}\#\cO_K^\times$ as usual,  $z_K\in E(K)$ is a Heegner point arising from a modular parametrization $\varphi:X_0(N)\rightarrow E$, 
\[
\log_{\omega_E}:E(K_\pp)\otimes\bZ_p\rightarrow\bZ_p
\] 
is the formal group logarithm associated with a N\'{e}ron differential $\omega_E\in\Omega^1(E/\bZ_{(p)})$, and $c_E\in\bZ$ is such that $\varphi^*(\omega_E)=c_E\cdot 2\pi if(\tau)d\tau$. 

Formula (\ref{eq:bdp}) has been a key ingredient in recent progress on the Birch and Swinnerton-Dyer conjecture. 
Most notably, 
for elliptic curves $E/\bQ$ with ${\rm rank}_\bZ E(\bQ)=1$ and $\#\Sha(E/\bQ)_{p^\infty}<\infty$, it was used by Skinner \cite{skinner} (for a suitable choice of $K$) to prove a converse to the celebrated theorem of Gross--Zagier and Kolyvagin, and for elliptic curves $E/\bQ$ with ${\rm ord}_{s=1}L(E,s)=1$, it was used by Jetchev--Skinner--Wan \cite{JSW} (again, for suitably chosen $K$) to prove under mild hypotheses the $p$-part of the Birch and Swinnerton-Dyer formula. 

For elliptic curves $E/\bQ$ 
satisfying ${\rm rank}_\bZ E(K)\geqslant 2$ and $\#\Sha(E/K)_{p^\infty}<\infty$, 
the Heegner point $z_K$ is torsion and formula (\ref{eq:bdp}) shows that $L_\pp(f)(\mathds{1})=0$, or equivalently, that $L_\pp(f)\in J$, where 
\[
J:={\rm ker}(\epsilon:\Lambda_{\hat{\cO}}\rightarrow\hat{\cO}) 
\]
is the augmentation ideal of $\Lambda_{\hat{\cO}}$. The 
conjectures formulated in this paper predict 
the largest power $J^\nu$ of $J$ in which $L_\pp(f)$ lives (i.e., the ``order of vanishing'' of $L_\pp(f)$ at $\mathds{1}$) in terms of the ranks of $E(\bQ)$ and $E(K)$, and a formula for the image of $L_\pp(f)$ in $J^\nu/J^{\nu+1}$ (i.e., the ``leading coefficient'' of $L_\pp(f)$ at $\mathds{1}$) in terms of arithmetic invariants of $E$. For the sake of illustration, in the remainder of this Introduction we concentrate on a weaker form of these conjectures, 
referring the reader to Section~3 
for their actual form.

Let $T=T_pE$ be the $p$-adic Tate module of $E$, and let $S_p(E/K)\subset{\rm H}^1(K,T)$ be the classical pro-$p$ Selmer group of $E$ fitting in the exact sequence
\[
0\rightarrow E(K)\otimes\bZ_p\rightarrow S_p(E/K)\rightarrow T_p\Sha(E/K)\rightarrow 0.
\] 
We shall assume that $\#\Sha(E/K)_{}<\infty$, so in particular $S_p(E/K)\simeq E(K)\otimes\bZ_p$. The work of Mazur--Tate \cite{mt-biext} (see also \cite{schneiderI}) produces a canonical symmetric $p$-adic height pairing
\[
h_p^{\tt MT}:S_p(E/K)\times S_p(E/K)\rightarrow (J/J^2)\otimes\bQ.
\]
By the $p$-parity conjecture \cite{nekovarII}, our assumptions also imply that the strict Selmer group 
\[
{\rm Sel}_{\rm str}(K,T):={\rm ker}\bigl\{S_p(E/K)\rightarrow E(K_\pp)\otimes\bZ_p\bigr\}
\]
has $\bZ_p$-rank $r-1$, where $r={\rm rank}_{\bZ}E(K)$. Assume that 
$S_p(E/K)$ is torsion-free (this holds if e.g. $E_p$ is irreducible as $G_K$-module), let $P_1,\dots,P_r$ be an integral basis for $E(K)\otimes\bQ$ and let $A$ be an endomorphism of $E(K)\otimes\bZ_p$ sending $P_1,\dots,P_r$ basis to a $\bZ_p$-basis $x_1,\dots,x_{r-1},y_\pp$ for $S_p(E/K)$ with $x_1,\dots,x_{r-1}$ generating ${\rm Sel}_{\rm str}(K,T)$. Set $t=\det(A)\cdot[E(K)\colon\bZ P_1+\cdots+\bZ P_r]$.
%
The following is a special case of our Conjecture~\ref{conj:BSD}. 

\begin{conj}\label{intro-conj}
Assume that $\Sha(E/K)$ is finite and that $E_p$ is irreducible as $G_K$-module, and let $r={\rm rank}_{\bZ}E(K)$. Then:
\begin{itemize}
\item[(i)] $L_{\pp}(f)\in J^{r-1}$.
\item[(ii)] Letting $\bar{L}_{\pp}(f)$ be the natural image of $L_\pp(f)$ in $J^{r-1}/J^r$, we have
\[
\bar{L}_\pp(f)=\biggl(\frac{1-a_p(E)+p}{p}\biggr)^2\cdot{\rm log}_{\omega_E}(y_\pp)^2\cdot{\rm Reg}_{\pp}\cdot t^{-2}\cdot\#\Sha(E/K)_{}\cdot\prod_{\ell\vert N}c_\ell^2,
\]
where ${\rm Reg}_\pp=\det(h_p^{\tt MT}(x_i,x_j))_{1\leqslant i,j\leqslant r-1}$ is the regulator of $h_p^{\tt MT}$ restricted to ${\rm Sel}_{\rm str}(K,T)$, and $c_\ell$ is the Tamagawa number of $E/\bQ_\ell$.
\end{itemize}
\end{conj}

\begin{rem}\label{rem:r=1}
Suppose ${\rm rank}_\bZ E(K)=1$. Under the assumptions of Conjecture~\ref{intro-conj} we then have  ${\rm Sel}_{\rm str}(K,T)=\{0\}$, ${\rm Reg}_\pp=1$ and $t=[E(K):\bZ.y_\pp]$, where $y_\pp$ is a  generator of $E(K)\otimes\bQ$ with $\log_{\omega_E}(y_\pp)\neq 0$. Thus in this case Conjecture~\ref{intro-conj} predicts that $L_\pp(f)(\mathds{1})\neq 0$, 
and by formula (\ref{eq:bdp}) 
the predicted expression for $L_\pp(f)(\mathds{1})$ is equivalent to the equality
\[
[E(K):\bZ.y_K]^2=u_K^2c_E^2\cdot\#\Sha(E/K)_{}\cdot\prod_{\ell\vert N}c_\ell^2,
\]
which is also predicted by the classical Birch and Swinnerton-Dyer conjecture for $E/K$ when combined with the Gross--Zagier formula (see \cite[Conj.~1.2]{gross-kolyvagin}). 
\end{rem}


The Iwasawa--Greenberg main conjecture \cite{Greenberg55} predicts that  $L_\pp(f)$ is a generator of the characteristic ideal of a certain $\Lambda$-torsion Selmer group $X_\pp$. Letting $F_\pp(f)\in\Lambda$ be a generator of this characteristic ideal, in this paper we prove the following result towards Conjecture~\ref{intro-conj}, which can be viewed simultaneously as a non-CM analogue of (resuls towards) Rubin's variant of the $p$-adic Birch and Swinnerton-Dyer conjecture (see \cite[\S{11}]{rubin-points}, \cite[Thm.~5]{rubin-BSD} and  \cite[Thm.~B]{agboola-I}) in the anticyclotomic setting, and an extension of the anticyclotomic control theorem of \cite[\S{3.3}]{JSW} to arbitrary ranks.


\begin{thm}[\emph{cf.} Theorem~\ref{thm:A}]\label{intro-thm}
Assume that $\Sha(E/K)_{p^\infty}$ is finite and that:
\begin{itemize} 
\item[{\rm (i)}] $\rho_{E,p}:G_\bQ\rightarrow{\rm Aut}_{\mathbf{F}_p}(E_p)$ is surjective. 
\item[{\rm (ii)}] $E_p$ is ramified at every prime $\ell\vert N$.
\item[{\rm (iii)}] $p\nmid\#E(\mathbf{F}_p)$.
\end{itemize}
Then $F_\pp(f)\in J^{r-1}$, where $r={\rm rank}_{\bZ}E(K)$, and letting $\bar{F}_{\pp}(f)$ be the natural image of $F_\pp(f)$ in $J^{r-1}/J^r$, we have 
\[
\bar{F}_\pp(f)=p^{-2}\cdot{\rm log}_{\omega_E}(y_\pp)^2\cdot{\rm Reg}_{\pp}\cdot t^{-2}\cdot\#\Sha(E/K)_{p^\infty}
\]
up to a $p$-adic unit. 
\end{thm}

\begin{rem}\label{rem:p-adic-unit}
Note that conditions (ii) and (iii) in Theorem~\ref{intro-thm} imply that the terms $1-a_p(E)+p$  and $c_\ell$ for $\ell\vert N$ are all $p$-adic units.
\end{rem}

Combined with the Iwasawa--Greenberg main conjecture for $L_\pp(f)$  (which is known under relatively mild hypotheses \cite{BCK-PRconj}), Theorem~\ref{intro-thm} can be parlayed in terms of this $p$-adic $L$-function, yielding our main result towards Conjecture~\ref{intro-conj} (or rather its refinement in Conjecture~\ref{conj:BSD}); see Corollary~\ref{cor:B}.


We end this Introduction with some comments about the need for the refinement of Conjecture~\ref{intro-conj} 
given by Conjecture~\ref{conj:BSD}.   
We continue to assume that $\#\Sha(E/K)_{p^\infty}<\infty$, and let 
\[
r^\pm:={\rm rank}_{\bZ}E(K)^\pm
\] 
be the rank of the $\pm$-eigenspaces $E(K)^\pm\subset E(K)$ under the action of complex conjugation, so 
\[
r={\rm rank}_\bZ E(K)=r^++r^-.
\] 
From the Galois-equivariance properties of $h_p^{\tt MT}$, one easily sees that ${\rm Reg}_\pp=0$ when $\vert r^+-r^-\vert >1$. These systematic degeneracies of the $p$-adic height pairing in the anticyclotomic setting (which are in sharp contrast with the expected non-degeneracy of the $p$-adic height pairing in the cyclotomic setting) 
were understood by Bertolini--Darmon \cite{BD-derived-DMJ, BD-derived-AJM} as giving rise to canonical \emph{derived} $p$-adic height pairings, in terms of which we will define a generalized $p$-adic regulator ${\rm Reg}_{{\rm der},\pp}$. This generalized regulator recovers 
${\rm Reg}_\pp$ when $\vert r^+-r^-\vert=1$, but provides extra information when $\vert r^+-r^-\vert>1$. More precisely, the expected ``maximal non-degeneracy'' of the anticyclotomic $p$-adic height pairing (as conjectured by Mazur and Bertolini--Darmon) leads to the prediction that ${\rm Reg}_{{\rm der},\pp}$ is a nonzero element in $J^{2\rho}/J^{2\rho+1}$, where 
\[
\rho={\rm max}\{r^+,r^-\}-1.
\]

Conjecture~\ref{conj:BSD} then predicts that $L_\pp(f)$ lands in $J^{2\rho}$ 
(note that $2\rho>r-1$ when $\vert r^+-r^-\vert>1$), and posits a formula for its natural image in $J^{2\rho}/J^{2\rho+1}$ in terms of ${\rm Reg}_{{\rm der},\pp}$. 
Our main result 
is the analogue of Theorem~\ref{intro-thm} for this refined conjecture. 

\begin{rem} 
As it will be clear to the reader, our conjectures are very closely related to the conjectures of Birch and Swinnerton-Dyer type formulated by Bertolini--Darmon \cite{BDmumford-tate} for certain ``Heegner distributions''. In fact, as application of our results on Conjecture~\ref{conj:BSD}, we will deduce under mild hypotheses the proof of an inequality in the ``rank part'' of their conjectures (see Corollary~\ref{cor:C}).
\end{rem}

The remainder of this paper is organized as follows. In Section~2, after defining the relevant Selmer groups and recalling the conjectures of Bertolini--Darmon,  
we formulate our conjectures of Birch and Swinnerton-Dyer type for the $p$-adic $L$-functions $L_\pp(f)$ and $\mathscr{L}_\pp(f)$. In Section~3, we state and prove our main results in the direction of these conjectures.  


\section{The conjectures}

\subsection{Selmer groups}\label{subsec:Sel}

We keep the notation from the Introduction. In particular, $K_\infty$ denotes the anticyclotomic $\bZ_p$-extension of $K$. For every $n$ we  write $K_n$ for the subextension of $K_\infty$ with 
\[
\Gamma_n:={\rm Gal}(K_n/K)\simeq\bZ/p^n\bZ.
\] 

Let $S$ be a finite set of places of $\bQ$ containing $\infty$ and the primes dividing $Np$, and for every finite extension $F/\bQ$ let $\mathfrak{G}_{F,S}$ be the Galois group over $F$ of the maximal extension of $F$ unramified outside the places above $S$. 
For each prime $\mathfrak{q}\in\{\pp,\overline{\pp}\}$ 
set
\[
\Sel_\qq(K_n,T):={\rm ker}\biggl\{{\rm H}^1(\mathfrak{G}_{K_n,S},T)\rightarrow\prod_{w\vert p, w\nmid\overline{\qq}}{\rm H}^1(K_{n,w},T)\biggr\}.
\]

Let $\Sel_\mathfrak{q}(K_n,E_{p^\infty})\subset{\rm H}^1(\mathfrak{G}_{K_n,S},E_{p^\infty})$ be the Selmer group cut out by the local conditions given by the orthogonal complement under local Tate duality of the subspaces cutting out $\Sel_\qq(K_n,T)$, and set
\[
\Sel_\qq(K_\infty,E_{p^\infty}):=\varinjlim_n\Sel_\qq(K_n,E_{p^\infty}).
\]
As is well-known (see e.g. \cite[Prop.~3.2]{greenberg-Coates}), $\Sel_{\mathfrak{q}}(K_\infty,E_{p^\infty})$ is a cofinitely generated $\Lambda$-module, i.e., its Pontryagin dual $\Sel_{\mathfrak{q}}(K_\infty,E_{p^\infty})^\vee$ is finitely generated over $\Lambda$.

\begin{conj}[Iwasawa--Greenberg main conjecture]\label{conj:IMC}
The module $\Sel_\pp(K_\infty,E_{p^\infty})$ is $\Lambda$-cotorsion and 
\[
{\rm char}_\Lambda(\Sel_\pp(K_\infty,E_{p^\infty})^\vee)\Lambda_{\hat{\cO}}=(L_\pp(f))
\]
as ideals in $\Lambda_{\hat{\cO}}$.
\end{conj}


The following lemma will be useful in the following. 
Let
\[
{\rm Sel}_{\rm str}(K,T):={\rm ker}\biggl\{{\rm H}^1(\mathfrak{G}_{K,S},T)\rightarrow\prod_{w}{\rm H}^1(K_{w},T)\biggr\}
\]
be the strict Selmer group, which is clearly contained in $S_p(E/K)$. 

\begin{lem}\label{lem:relation}
Assume that $\Sha(E/K)_{p^\infty}$ is finite. Then 
\[
\Sel_{\pp}(K,T)={\rm Sel}_{\rm str}(K,T)=\Sel_{\overline\pp}(K,T).
\]
In particular, $\Sel_\pp(K,T)$ and $\Sel_{\overline\pp}(K,T)$ are both contained in $S_p(E/K)$ and have $\bZ_p$-rank $r-1$, where $r={\rm rank}_\bZ E(K)$. 
\end{lem}

\begin{proof}
By our assumption on $\Sha(E/K)$, hypothesis (\ref{eq:Heeg}) and the $p$-parity conjecture \cite{nekovarII} imply that $r:={\rm rank}_{\bZ}E(K)$ is odd, so in particular $r>0$. 
Thus the image of restriction map
\begin{equation}\label{eq:res-p}
S_p(E/K)\rightarrow\prod_{w\vert p}E(K_w)\otimes\bZ_p
\end{equation}
has $\bZ_p$-rank one, and the result follows from \cite[Lem.~2.3.2]{skinner}.
\end{proof}






\subsection{Conjectures of Bertolini--Darmon}\label{subsec:conj-BD}

In this section, we recall some of the conjectures of Birch--Swinnerton-Dyer type formulated by Bertolini--Darmon in \cite{BDmumford-tate}. These conjectures will guide our formulation in $\S\ref{subsec:BSDconj}$ of analogous statements for the $p$-adic $L$-functions $L_\pp(f)$ and $\mathscr{L}_\pp(f)$ of Bertolini--Darmon--Prasanna. 

As in the Introduction, we assume that the elliptic curve $E/\bQ$ has good ordinary reduction at $p>2$ and that $K$ is an imaginary quadratic field of discriminant $D_K$ prime to $Np$ in which $p=\pp\ppbar$ splits. 
However, rather than hypothesis (\ref{eq:Heeg}) from the Introduction, we assume that writing $N$ as the product 
\[
N=N^+N^-,
\]
with $N^+$ (resp. $N^-$) divisible only by primes which are split (resp. inert) in $K$, we have 
\begin{equation}\label{eq:gen-H}
\textrm{$N^-$ is the squarefree product of an even number of primes.}\tag{gen-H}
\end{equation}

This condition still guarantees that the root number of $E/K$ is $-1$, as well as the presence of Heegner points on $E$ defined over the different layers of the anticyclotomic $\bZ_p$-extension $K_\infty/K$.

More precisely, let $X_{N^+,N^-}$ be the Shimura curve (with the cusps added when $N^-=1$, so $X_{N,1}=X_0(N)$) attached to the quaternion algebra $B/\bQ$ of discriminant $N^-$ and an Eichler order $R\subset B$ of level $N^+$. The curve $X_{N^+,N^-}$ has a canonical model over $\bQ$, and we let $J(X_{N^+,N^-})_{/\bQ}$ denote its Jacobian. By \cite{BCDT}, there is a modular parametrization 
\[ 
\varphi:X_0(N)\rightarrow E.
\]
This induces a map $J(X_0(N))\rightarrow E$ by Albanese functoriality, which by the Jacquet--Langlands correspondence together with Faltings' isogeny theorem gives rise to a map 
\begin{equation}\label{eq:param}
\varphi_*:J(X_{N^+,N^-})\rightarrow E.
\end{equation}
Similarly as in \cite[p.~425]{BDmumford-tate}, after possibly changing $E$ within its isogeny class, we assume that $E$ is an optimal quotient of $J(X_{N^+,N^-})$, meaning that the kernel of $(\ref{eq:param})$ is connected.

When $N^-\neq 1$, lacking the existence of a natural rational base point on $X_{N^+,N^-}$, we choose  an auxiliary prime $\ell_0$ 
and consider (following \cite[\S{4.2}]{JSW}) the embedding
\begin{equation}\label{eq:ell0}
\iota_{N^+,N^-}:X_{N^+,N^-}\rightarrow J(X_{N^-,N^-})
,\quad x\mapsto(T_{\ell_0}-\ell_0-1)[x].
\end{equation}

Let $K[c]$ be the ring class field of $K$ of conductor $c$. Then for every $c$ prime to $ND_K$, there are CM points $h_c\in X_{N^+,N^-}(K[c])$ (as described in e.g. \cite[Prop.~1.2.1]{howard-PhD-II}) satisfying the relations
\begin{equation}\label{eq:norm-rel}
\begin{aligned}
{\rm Norm}_{K[c\ell]/K[c]}(h_c)=\left\{
\begin{array}{ll}
T_\ell\cdot h_c&\textrm{if $\ell\nmid c$ is inert in $K$,}\\
T_\ell\cdot h_c-\sigma_\ell h_c^{}-\sigma_\ell^*h_c^{} &\textrm{if $\ell\nmid c$ splits in $K$,}\\
T_\ell\cdot h_c-h_{c/\ell} &\textrm{if $\ell\mid c$,}
\end{array}
\right.
\end{aligned}
\end{equation}
where $\sigma_{\ell}$ and $\sigma_\ell^*$ denote the Frobenius elements of the primes in $K$ above $\ell$. 
Assume from now on that $E_p$ is irreducible as a $G_\bQ$-module, and choose the prime $\ell_0$ in $(\ref{eq:ell0})$ so that $a_{\ell_0}(E)-\ell_0-1\not\in p\bZ$. Define $y_n\in E(K[p^n])\otimes\bZ_p$ by
\[
y_n:=\frac{1}{a_{\ell_0}(E)-\ell_0-1}\cdot\varphi_*(\iota_{N^+,N^-}(h_{p^n})),
\]
and letting $\alpha_p$ be the $p$-adic unit root of the polynomial $X^2-a_p(E)X+p$, define the \emph{regularized Heegner point} of conductor $p^n$ by
\begin{align*}
z_n&:=\frac{1}{\alpha^{n}}\cdot y_n-\frac{1}{\alpha^{n+1}}\cdot y_{n-1},\quad\textrm{if $n\geqslant 1$},\\
z_0&:=u_K^{-1}\cdot(1-(\sigma_p+\sigma_{p}^*)\alpha^{-1}+\alpha^{-2})\cdot y_0.
\end{align*}
Then one immediately checks from $(\ref{eq:norm-rel})$ that the points $z_n$ are norm-compatible. For each $n\geqslant 0$, we then set
\begin{equation}\label{def:regularised}
\mathbf{z}_n:={\rm Norm}_{K[p^m]/K_n}(z_m),
\end{equation}
where $m\gg 0$ is such that $K_n\subset K[p^m]$, and letting  $Z_p:=E(K_\infty)\otimes\bZ_p$ we define $\theta_n\in Z_p[\Gamma_n]$ by
\[
\theta_n:=\sum_{\sigma\in\Gamma_n}\mathbf{z}_n^\sigma\otimes\sigma^{-1}.
\]
These elements are compatible under the natural projections $Z_p[\Gamma_{n+1}]\rightarrow Z_p[\Gamma_n]$, and in the limit they define the ``Heegner distribution''
\begin{equation}\label{def:Heeg-dist}
\theta=\theta_\infty:=\varprojlim_n\theta_n\in Z_p[[\Gamma_\infty]]. 
\end{equation}

Let $J$ be the augmentation ideal of $\Lambda=\bZ_p[[\Gamma_\infty]]$, and define the \emph{order of vanishing} of $\theta$ by
\[
{\rm ord}_J\theta:=\max\left\{\rho\in\bZ_{\geqslant 0}\;\colon\;\theta\in Z_p\otimes_{\bZ_p} J^\rho\right\}.
\]
The work of Cornut--Vatsal \cite{CV-dur} implies that $\theta$ is a nonzero element in $Z_p[[\Gamma_\infty]]$, and so its order of vanishing is well-defined. 
%

The following conjecture is the ``indefinite case'' of \cite[Conj.~4.1]{BDmumford-tate}, where we let $E(K)^\pm$ be the $\pm$-eigenspaces of $E(K)$ under the action of complex conjugation. 

\begin{conj}[Bertolini--Darmon]\label{conj:BDconj-rank}
We have
\[
{\rm ord}_J\theta=\max\{r^+,r^-\}-1,
\]
where $r^{\pm}:={\rm rank}_{\bZ}E(K)^\pm$. 
\end{conj}

Let $\theta^*$ denote the image of $\theta$ under the involution of $Z_p[[\Gamma_\infty]]$ given by $\gamma\mapsto\gamma^{-1}$ for $\gamma\in\Gamma_\infty$, and set
\[
\mathscr{L}:=\theta\otimes\theta^*\in Z_p^{\otimes 2}[[\Gamma_\infty]].
\] 

\begin{lem}\label{lem:incl}
Let $\rho={\rm ord}_J\theta$. Then the natural image $\bar{\mathscr{L}}$ 
of $\mathscr{L}$ in $Z_p^{\otimes 2}\otimes_{\bZ_p}(J^{2\rho}/J^{2\rho+1})$ is contained in the image of the map
\[
E(K)^{\otimes 2}\otimes(J^{2\rho}/J^{2\rho+1})\rightarrow Z_p^{\otimes 2}\otimes_{\bZ_p}(J^{2\rho}/J^{2\rho+1}).
\]
\end{lem}

\begin{proof}
This follows from the fact that the natural image 
of $\theta$ in $Z_p\otimes_{\bZ_p}(J^{\rho}/J^{\rho+1})$ is fixed by $\Gamma_\infty$ (see \cite[Lem.~2.14]{BDmumford-tate}).
\end{proof}

Let $r={\rm rank}_{\bZ}E(K)$. Since clearly 
\[
2(\max\{r^+,r^-\}-1)\geqslant r-1,
\] 
by Lemma~\ref{lem:incl} we see that Conjecture~\ref{conj:BDconj-rank} predicts in particular the inclusion $\bar{\mathscr{L}}\in E(K)^{\otimes 2}\otimes(J^{r-1}/J^r)$. The conjectures of Bertolini--Darmon also predict an expression for $\bar{\mathscr{L}}$ 
in terms of the following ``enhanced'' regulator associated to the Mazur--Tate anticyclotomic $p$-adic height pairing
\[
h_p^{\tt MT}:E(K)\times E(K)\rightarrow(J/J^2)\otimes\bQ.
\]

\begin{defn}\label{def:enhanced}
Let $P_1,\dots,P_r$ be a basis for $E(K)/E(K)_{\rm tors}$ and set $t'=[E(K)\colon\bZ P_1+\cdots+\bZ P_r]$. 
The \emph{enhanced regulator} $\widetilde{\rm Reg}$ is the element of $E(K)^{\otimes 2}\otimes_{}(J^{r-1}/J^r)\otimes\bQ$ defined by
\[
\widetilde{\rm Reg}:=\frac{1}{t'^2}\sum_{i,j=1}^r(-1)^{i+j}P_i\otimes P_j\otimes R_{i,j},
\]
where 
$R_{i,j}$ is the $(i,j)$-minor of the matrix $(h_p^{\tt MT}(P_i,P_j))_{1\leqslant i,j\leqslant r}$.
\end{defn}

The next remark will be important in the following.

\begin{rem}\label{rem:isotropic}
The non-trivial automorphism $\tau\in{\rm Gal}(K/\bQ)$ acts as multiplication by $-1$ on $\Gamma_\infty$. Viewing $h_p^{\tt MT}$ as valued in $\Gamma_\infty\otimes_{}\bQ$ via the  natural identification $J/J^2\simeq\Gamma_\infty$, the Galois-equivariance of $h_p^{\tt MT}$ implies that
\[
h_p^{\tt MT}(\tau x,\tau y)=h_p^{\tt MT}(x,y)^\tau=-h_p^{\tt MT}(x,y).
\]
It follows that the $\tau$-eigenspaces $E(K)^\pm$ are isotropic for $h_p^{\tt MT}$, and so the null-space of $h_p^{\tt MT}$ 
has rank at least $\vert r^+-r^-\vert$ (which should always be positive, since by (\ref{eq:gen-H}) the rank $r=r^++r^-$ should be odd). 
\end{rem}

The following is the ``non-exceptional case\footnote{meaning that $E/K$ has good ordinary or non-split multiplicative reduction at every prime above $p$}'' of \cite[Conj.~4.5]{BDmumford-tate}. 

\begin{conj}[Bertolini--Darmon]\label{conj:BDconj-weakLC}
Let $\bar{\mathscr{L}}$ be the natural image of $\mathscr{L}$ in $E(K)^{\otimes 2}\otimes_{}(J^{r-1}/J^{r})$. Then
\[
\bar{\mathscr{L}}=\biggl(\frac{1-a_p(E)+p}{p}\biggr)^2\cdot\widetilde{\rm Reg}\cdot\#\Sha(E/K)_{}\cdot\prod_{\ell\vert N^+}c_\ell^2,
\]
where 
$c_\ell$ is the 
Tamagawa number of $E/\bQ_\ell$.
\end{conj}

As noted in \cite[p.~447]{BDmumford-tate}, when $\vert r^+-r^-\vert>1$   
Conjecture~\ref{conj:BDconj-weakLC} reduces to the prediction ``$0=0$''. Indeed, $2(\max\{r^+,r^-\}-1)$ is then strictly larger than $r-1$, and so by Conjecture~\ref{conj:BDconj-rank} 
the image of $\mathscr{L}$ in $E(K)^{\otimes 2}\otimes(J^{r-1}/J^r)$ should vanish, while on the other hand by the isotropy of $E(K)^\pm$ under $h_p^{\tt MT}$ all the minors $R_{i,j}$ in the definition of $\widetilde{\rm Reg}$, and hence $\widetilde{\rm Reg}$ itself, 
 also vanish (see \cite[Lem.~3.2]{BDmumford-tate}).
As explained below, a refinement of Conjecture~\ref{conj:BDconj-weakLC}, predicting a formula for the natural image of $\mathscr{L}$ in $E(K)^{\otimes 2}\otimes(J^{2\rho}/J^{2\rho+1})$, which should be thought of as the ``leading coefficient'' of $\mathscr{L}$, can be given in terms of the \emph{derived} $p$-adic height pairings introduced by Bertolini--Darmon \cite{BD-derived-DMJ, BD-derived-AJM}.

\begin{rem}\label{eq:rem-refined}
Such refinement of Conjecture~\ref{conj:BDconj-weakLC} seems to not have been explicitly stated in the literature. Even though the formulation of such refinement appears to be quite clear in light of the conjectures explicitly stated in \cite{BDmumford-tate} and \cite{BD-derived-AJM}, any inaccuracies in the conjectures below should be  blamed only on the authors of this paper. 
\end{rem}

Assume from now on that $\Sha(E/K)_{p^\infty}$ is finite and that:
\begin{itemize}
\item[(i)] $\bar{\rho}_{E,p}:G_\bQ\rightarrow{\rm Aut}_{\mathbf{F}_p}(E_p)$ is surjective. 
\item[(ii)] $p\nmid\#E(\mathbf{F}_p)$. 
\item[(iii)] $E_p$ is ramified at every prime $\ell\vert N$.
\end{itemize}
Note that (ii) amounts to the condition $a_p(E)\not\equiv{1}\pmod{p}$,
and condition (i) implies that $E$ has no CM. In  particular, these assumptions imply that $S_p(E/K)\simeq E(K)\otimes\bZ_p$ is a free $\bZ_p$-module of rank $r$, and the pair $(E,K)$ is ``generic'' in the terminology of \cite{mazur-ICM84}. 

By \cite[\S{2.4}]{BD-derived-AJM}, 
there is a filtration
\begin{equation}\label{eq:fil}
S_p(E/K)=\Sar_p^{(1)}\supset\Sar_p^{(2)}\supset\cdots\supset \Sar_p^{(p)},
\end{equation}
and a sequence of ``derived $p$-adic height pairings'' 
\[
h_p^{(k)}:\Sar_p^{(k)}\times\Sar_p^{(k)}\rightarrow (J^k/J^{k+1})\otimes_{}\bQ,\quad\textrm{for $1\leqslant k\leqslant p-1$,}
\] 
such that $\Sar_p^{(k+1)}$ is the null-space of $h_p^{(k)}$, with $h_p^{(1)}=h_p^{\tt MT}$. 
By Remark~\ref{rem:isotropic}, 
$\Sar_p^{(2)}$ has $\bZ_p$-rank at least $\vert r^+-r^-\vert$, and by construction the subspace of universal norms
\[
US_p(E/K):=\bigcap_{n\geqslant 1}{\rm cor}_{K_n/K}(S_p(E/K_n))
\]
is contained in the null-space of all $h_p^{(k)}$. The work of Cornut--Vatsal implies that $US_p(E/K)\simeq\Z_p$.

The expected ``maximal non-degeneracy'' of $h_p^{\tt MT}$ predicts the following  
(see \cite[Conj.~3.3, Conj.~3.8]{BD-derived-AJM}).

\begin{conj}[Mazur, Bertolini--Darmon]\label{conj:mazur-BD}
Under the above hypotheses we have
\[
{\rm rank}_{\bZ_p}\Sar_p^{(k)}=\left\{
\begin{array}{ll}
\vert r^+-r^-\vert &\textrm{if $k=2$,}\\[0.1cm]
1&\textrm{if $k\geqslant 3$,}
\end{array}
\right.
\]
and in fact $\Sar_p^{(3)}=US_p(E/K)$. 
\end{conj}

By construction, 
the successive quotients $\Sar_p^{(k)}/\Sar_p^{(k+1)}$
are free $\Z_p$-modules, say
\begin{equation}\label{eq:e_k}
\Sar_p^{(k)}/\Sar_p^{(k+1)}\simeq\Z_p^{e_k},
\end{equation}
and Conjecture~\ref{conj:mazur-BD} predicts in particular that 
\[
e_1=2\min\{r^+,r^-\},\quad e_2=\vert r^+-r^-\vert-1,
\] 
and $e_k=0$ for all $k\geqslant 3$. 


Using derived $p$-adic height pairings, one can define an enhanced $p$-adic regulator extending Definition~\ref{def:enhanced}. Assume that $\Sar_p^{(p)}=US_p(E/K)$ 
(as Conjecture~\ref{conj:mazur-BD} predicts in particular). Let $P_1,\dots, P_r$ be an integral basis for $E(K)\otimes\bQ$, and let $A\in M_{n}(\bZ_p)$ be an endomorphism of $S_p(E/K)$ sending $P_1,\dots,P_r$ to a $\bZ_p$-basis $x_1,\dots, x_r$ for $S_p(E/K)$ compatible with the filtration (\ref{eq:fil}), so for $1\leqslant k\leqslant p-1$ the projection of say $x_{h_k+1},\dots,x_{h_k+e_k}$ to $\Sar_p^{(k)}/\Sar_p^{(k+1)}$ is a $\bZ_p$-basis for $\Sar_p^{(k)}/\Sar_p^{(k+1)}$ and $y:=x_r$ generates $US_p(E/K)$. Set $t=\det(A)\cdot[E(K)\colon\bZ P_1+\cdots+\bZ P_r]$.

\begin{defn}\label{def:der-enhanced}
Let $\varrho:=\sum_{k=1}^{p-1}ke_k$. The \emph{derived enhanced regulator} $\widetilde{\rm Reg}_{\rm der}$ is the element of $E(K)^{\otimes 2}\otimes_{}(J^{\varrho}/J^{\varrho+1})\otimes\bQ$  defined by
\[
\widetilde{\rm Reg}_{\rm der}:=
t^{-2}\cdot(y\otimes y)\otimes
\prod_{k=1}^{p-1}R^{(k)},
\]
where $R^{(k)}=\det(h_p^{(k)}(x_i,x_j))_{h_k+1\leqslant i,j\leqslant h_k+e_k}$. 
\end{defn}

The relation between $\widetilde{\rm Reg}_{\rm der}$ and $\widetilde{\rm Reg}$ is readily described.

\begin{lem}\label{lem:der-extends}
Assume Conjecture~\ref{conj:mazur-BD}. 
If $\vert r^+-r^-\vert=1$, then 
$\widetilde{\rm Reg}_{\rm der}=\widetilde{\rm Reg}$. 
\end{lem}

\begin{proof}
By our running assumption that $\#\Sha(E/K)_{p^\infty}<\infty$, we may view $h_p^{\tt MT}$ as defined on $S_p(E/K)$. Denote by $R_{i,j}'$ the $(i,j)$-minor of the matrix $(h_p^{\tt MT}(x_i,x_j))_{1\leqslant i,j\leqslant r}$. Since universal norms are in the null-space of $h_{p}^{\tt MT}$, we find that
\begin{align*}
\widetilde{\rm Reg}
&=t^{-2}\sum_{1\leqslant i,j\leqslant r}(-1)^{i+j}x_i\otimes x_j\otimes R'_{i,j}\\
&=t^{-2}\cdot(y\otimes y)\otimes R'_{r,r},
\end{align*}
noting that for $(i,j)\neq (r,r)$ the minor $R'_{i,j}$ is the determinant of a matrix having either a row or a column consisting entirely of zeroes. Since 
our assumptions together with $(\ref{eq:e_k})$ imply that 
\[
\Sar_p^{(2)}=\Sar_p^{(3)}=\cdots=US_p(E/K),
\] 
we conclude that  
$\prod_{k=1}^{p-1}R^{(k)}=R^{(1)}=
R_{r,r}'$, 
hence the result.
\end{proof}

In general, Conjecture~\ref{conj:mazur-BD} predicts that $\widetilde{\rm Reg}_{\rm der}$ is a nonzero element in 
$E(K)^{\otimes 2}\otimes(J^{\varrho}/J^{\varrho+1})\otimes\bQ$, where
\[
\varrho\overset{}=e_1+2e_2=2\min\{r^+,r^-\}+2(\vert r^+-r^-\vert-1)=2(\max\{r^+,r^-\}-1),
\]
which as already noted is strictly larger than $r-1$ when $\vert r^+-r^-\vert>1$. Thus, by Lemma~\ref{lem:der-extends} the following refines Conjecture~\ref{conj:BDconj-weakLC}.

\begin{conj}[Bertolini--Darmon]\label{conj:BD-general}
Under the above hypotheses we have
\[
{\rm ord}_J\mathscr{L}=2(\max\{r^+,r^-\}-1),
\]
and letting $\bar{\mathscr{L}}$ be the natural image of $\mathscr{L}$ in $E(K)^{\otimes 2}\otimes(J^{2\rho}/J^{2\rho+1})$, where $\rho=\max\{r^+,r^-\}-1$, we have
\[
\bar{\mathscr{L}}=\biggl(\frac{1-a_p(E)+p}{p}\biggr)^2\cdot\widetilde{\rm Reg}_{\rm der}\cdot\#\Sha(E/K)_{}\cdot\prod_{\ell\vert N^+}c_\ell^2.
\]
\end{conj}

It is also possible to formulate a leading term formula for the Heegner distribution $\theta$, refining the ``non-exceptional case'' of \cite[Conj.~4.6]{BDmumford-tate}. 


The subspace of universal norms $US_p(E/K)$ is stable under the action of ${\rm Gal}(K/\bQ)$, and therefore is contained in one of the $\tau$-eigenspaces $S_p(E/K)^\pm$. 

\begin{lem}\label{lem:sign}
Assume Conjecture~\ref{conj:mazur-BD}. Letting ${\rm sign}\;US_p(E/K)$ be the sign of the $\tau$-eigenspace where $US_p(E/K)$ is contained, we have 
\[
{\rm sign}\;US_p(E/K)=\left\{
\begin{array}{ll}
1&\textrm{if $r^+>r^-$},\\[0.1cm]
-1&\textrm{if $r^->r^+$.}
\end{array}
\right.
\]
In other words, $US_p(E/K)$ is contained in the larger of the $\tau$-eigenspaces $S_p(E/K)^\pm$.
\end{lem}

\begin{proof}
Viewing $h_p^{\tt MT}$ as defined on $S_p(E/K)$, Conjecture~\ref{conj:mazur-BD} predicts that the restriction
\[
h_p^{\tt MT}:S_p(E/K)^+\times S_p(E/K)^-\rightarrow(J/J^2)\otimes\bQ
\]
is either left non-degenerate or right non-degenerate, depending on which of the $\tau$-eigenspaces $S_p(E/K)^\pm\subset S_p(E/K)$ is larger. Since the universal norms are contained in the null-space of $h_p^{\tt MT}$, it follows that $US_p(E/K)$ is contained in the $\tau$-eigenspace of larger rank. 
\end{proof}

\begin{rem}
The conclusion of Lemma~\ref{lem:sign} is predicted by the ``sign conjecture'' of Mazur--Rubin \cite[Conj.~4.8]{MR-kato}, and the fact that it follows from Conjecture~\ref{conj:mazur-BD} 
was already observed by them.
\end{rem}

Let $s:=\min\{r^+,r^-\}$ and recall that Conjecture~\ref{conj:mazur-BD} predicts $e_1:={\rm rank}_{\bZ_p}\Sar_p^{(1)}/\Sar_p^{(2)}=2s$.  
Order the first $2s$ elements of the basis $x_1,\dots,x_r$ for $S_p(E/K)$ so that $x_1=:y_1^+,\dots,x_s=:y_s^+$ belong to  $S_p(E/K)^+$ and $x_{s+1}=:y_1^-,\cdots,x_{2s}=:y_s^-$ belong to $S_p(E/K)^-$.

\begin{lem}\label{lem:square}
We have
\[
R^{(1)}=
-(\det(h_p^{\tt MT}(y_i^+,y_j^-)_{1\leqslant i,j\leqslant s})^2.
\]
\end{lem}

\begin{proof}
This is immediate from the isotropic property of $S_p(E/K)^\pm$ under the pairing $h_p^{\tt MT}$ (see Remark~\ref{rem:isotropic}). 
\end{proof}

Thus $R^{(1)}$ is essentially a square. On the other hand, since for even values of $k$ the pairing $h_p^{(k)}$ is alternating (see part~(1) of \cite[Thm.~2.18]{BD-derived-AJM}), we have
\[
R^{(2)}={\rm pf}(h_p^{(2)}(x_i,x_j)_{e_1+1\leqslant i,j\leqslant e_1+e_2})^2,
\]
where ${\rm pf}(M)$ denotes the Pfaffian of the matrix $M$. 
This motivates the following definition of a square-root of the regulator $\widetilde{\rm Reg}_{\rm der}$ in Definition~\ref{def:der-enhanced}.

\begin{defn}\label{def:enhanced-der-sqrt}
Assume Conjecture~\ref{conj:mazur-BD}. The \emph{square-root derived enhanced regulator} is the element of $E(K)\otimes(J^\rho/J^{\rho+1})\otimes\bQ$, where $\rho=\max\{r^+,r^-\}-1$, defined by  
\[
\widetilde{\rm Reg}_{\rm der}^{1/2}:=t^{-1}\cdot y\otimes
(\det(h_{p}^{\tt MT}(y_i^+,y_j^-)_{1\leqslant i,j\leqslant s})\cdot{\rm pf}(h_p^{(2)}(x_i,x_j)_{e_1+1\leqslant i,j\leqslant e_1+e_2}).
\]
Note that this is only well-defined up to sign. 
\end{defn} 


The following refines \cite[Conj.~4.6]{BDmumford-tate} in the cases where $\vert r^+-r^-\vert>1$, and complements Conjecture~\ref{conj:BDconj-rank} with a leading coefficient formula.

\begin{conj}[Bertolini--Darmon]\label{conj:BD-sqrt}
We have 
\[
{\rm ord}_J\theta=\max\{r^+,r^-\}-1,
\]
and letting $\bar{\theta}$ be the natural image of $\theta$ in $(E(K_\infty)\otimes J^{\rho}/J^{\rho+1})^{\Gamma_\infty}$, where $\rho=\max\{r^+,r^-\}-1$, the following equality holds
\[
\bar{\theta}=\pm\biggl(\frac{1-a_p(E)+p}{p}\biggr)\cdot\widetilde{\rm Reg}^{1/2}_{\rm der}\cdot\sqrt{\#\Sha(E/K)_{}}
\cdot\prod_{\ell\vert N^+}c_\ell.
\]
\end{conj}


\subsection{Conjectures for $L_\pp(f)$ and $\mathscr{L}_\pp(f)$}\label{subsec:BSDconj}

We keep the hypotheses on the triple $(E,p,K)$ from $\S\ref{subsec:conj-BD}$ (in particular, we assume $\#\Sha(E/K)_{p^\infty}<\infty$), and assume in addition that hypothesis (\ref{eq:Heeg}) from the Introduction (rather than the more general (\ref{eq:gen-H})) holds.


\begin{rem}
The assumption that $p=\pp\ppbar$ splits in $K$ will be essential in what follows, so that the $p$-adic $L$-function $L_\pp(f)$ can be constructed as an element in $\Lambda_{\hat{\cO}}$ (\emph{cf.} \cite{kriz-PhD, AI-non-split} in the case when $p$ is non-split in $K$). On the other hand, it should not be difficult to extend the construction of $L_\pp(f)$ in \cite{cas-hsieh1} under the generalized Heegner hypothesis (\ref{eq:gen-H}) considered in $\S\ref{subsec:conj-BD}$.
\end{rem}

By Lemma~\ref{lem:relation}, the Selmer groups $\Sel_\pp(K,T)$ and $\Sel_{\bar{\pp}}(K,T)$ are both contained in $S_p(E/K)$ and they agree with the kernel ${\rm Sel}_{\rm str}(K,T)$ of the restriction map $(\ref{eq:res-p})$. 
Thus we can consider the pairing
\[
h_\pp:\Sel_{\pp}(K,T)\times\Sel_{\ppbar}(K,T)\rightarrow(J/J^2)\otimes\bQ
\]
obtained by restricting $h_p^{\tt MT}$. The filtration in $(\ref{eq:fil})$ induces a filtration
\begin{equation}\label{eq:filp}
\Sel_\pp(K,T)=\fSar_\pp^{(1)}\supset\fSar_\pp^{(2)}\supset\cdots\supset\fSar_\pp^{(p)}
\end{equation}
defined by $\fSar_\pp^{(k)}:=\Sar_p^{(k)}\cap\Sel_\pp(K,T)$, with the filtered pieces equipped with corresponding derived $p$-adic height pairing
\[
h_\pp^{(k)}:\fSar_\pp^{(k)}\times\fSar_\pp^{(k)}\rightarrow(J^k/J^{k+1})\otimes\bQ
\]
obtained from $h_p^{(k)}$ by restriction.

Assume that $\Sar_p^{(p)}=US_p(E/K)$ 
and that $\Sel_\pp(K_\infty,E_{p^\infty})$ is $\Lambda$-cotorsion. Then $\varprojlim_n\Sel_\pp(K_n,T)$ vanishes (see e.g. \cite[Lem.~A.3]{cas-BF}), and therefore the subspace of universal norms $U\Sel_\pp(K,T)\subset\Sel_\pp(K,T)$ is trivial. 
It follows that 
\[
US_p(E/K)\cap{\rm Sel}_{\rm str}(K,T)=\{0\},
\] 
and so $\log_{\omega_E}(y)\neq 0$ for any generator $y\in US_p(E/K)$. Thus the first $r-1$ elements in the basis $x_1,\dots,x_r$ for $S_p(E/K)$ chosen for the definition of $\widetilde{\rm Reg}_{\rm der}$ yield a basis for $\Sel_\pp(K,T)$ adapted to the filtration $(\ref{eq:filp})$, with the image of $x_{h_k+1},\dots,x_{h_k+e_k}$ in 
\begin{equation}\label{eq:e_k-p}
\Sar_p^{(k)}/\Sar_p^{(k+1)}\simeq\fSar_\pp^{(k)}/\fSar_\pp^{(k+1)}\simeq\bZ_p^{e_k}
\end{equation} 
giving a basis for $\fSar_\pp^{(k)}/\fSar_\pp^{(k+1)}$. Then the partial regulators  of Definition~\ref{def:enhanced} can be rewritten as
\begin{equation}\label{eq:partialp}
R^{(k)}=\det(h_\pp^{(k)}(x_i,x_j))_{h_k+1\leqslant i,j\leqslant h_k+e_k}
={\rm disc}(h_\pp^{(k)}\vert_{\fSar_\pp^{(k)}/\fSar_\pp^{(k+1)}}),
\end{equation}
which we shall denote by $\mathcal{R}_\pp^{(k)}$ in the following.

We can now define the $p$-adic regulator appearing in the leading term formula of our $p$-adic Birch and Swinnerton-Dyer conjecture for $L_\pp(f)$. The map $\log_{\omega_E}$ gives rise to a  map
\[
{\rm Log}_\pp:(E(K)\otimes\bZ_p)^{\otimes 2}\rightarrow (E(K_\pp)\otimes \bZ_p)^{\otimes 2}\xrightarrow{\log_{\omega_E}^{\otimes 2}}\bZ_p\otimes\bZ_p\rightarrow\bZ_p,
\]
where the last arrow is given by multiplication. Choose a basis $x_1,\dots,x_{r-1},x_r$ as before, with $x_r=y_\pp$ given by a generator for $US_p(E/K)$ with $p^{-1}\log_{\omega_E}(y_\pp)\not\equiv 0\pmod{p}$.

\begin{defn}\label{def:regp}
The \emph{derived regulator} ${\rm Reg}_{\pp,{\rm der}}$ is defined by
\begin{align*}
{\rm Reg}_{\pp,{\rm der}}:={\rm Log}_\pp(\widetilde{\rm Reg}_{\rm der})
&=t^{-2}\cdot\log_{\omega_E}(y_\pp)^2\cdot\prod_{k=1}^{p-1}\mathcal{R}_\pp^{(k)}.
\end{align*}
\end{defn}

Note that ${\rm Reg}_{\pp,{\rm der}}$ is an element in $(J^{\varrho}/J^{\varrho+1})\otimes\bQ$, where $\varrho=\sum_{k=1}^{p-1}ke_k$, and Conjecture~\ref{conj:mazur-BD} predicts the equality $\varrho=2(\max\{r^+,r^-\}-1)$.



\begin{conj}\label{conj:BSD}
We have
\[
{\rm ord}_JL_\pp(f)=2(\max\{r^+,r^-\}-1),
\]
and letting $\bar{L}_\pp(f)$ be the natural image of $L_\pp(f)$ in $J^{2\rho}/J^{2\rho+1}$, where $\rho=\max\{r^+,r^-\}-1$, the following equality holds
\[
\bar{L}_\pp(f)=\biggl(\frac{1-a_p(E)+p}{p}\biggr)^2\cdot{\rm Reg}_{\pp,{\rm der}}\cdot\#\Sha(E/K)\cdot\prod_{\ell\vert N}c_\ell^2.
\]
\end{conj}

Similarly as in $\S\ref{subsec:conj-BD}$, 
we can also formulate a version of Conjecture~\ref{conj:BSD} for the  ``square-root'' $p$-adic $L$-function $\mathscr{L}_\pp(f)$. Assume Conjecture~\ref{conj:mazur-BD}, so following Definition~\ref{def:enhanced-der-sqrt} we can define the derived square-root regulator ${\rm Reg}_{\pp,{\rm der}}^{1/2}$ by
\[
{\rm Reg}_{\pp,{\rm der}}^{1/2}:=t^{-1}\cdot\log_{\omega_E}(y_\pp)\cdot(\det(h_\pp(y_i^+,y_j^-)_{1\leqslant i,j\leqslant s})\cdot{\rm pf}(h_p^{(2)}(x_i,x_j)_{e_1+1\leqslant i,j\leqslant e_1+e_2}).
\]
As before, note that ${\rm Reg}_{\pp,{\rm der}}^{1/2}$ is only well-defined up to sign, and is contained in $(J^{\rho}/J^{\rho+1})\otimes\bQ$, where $\rho=\max\{r^+,r^-\}-1$.

\begin{conj}\label{conj:BSD-sqrt}
We have
\[
{\rm ord}_J\mathscr{L}_\pp(f)=\max\{r^+,r^-\}-1,
\]
and letting $\bar{\mathscr{L}}_\pp(f)$ be the natural image of $\mathscr{L}_\pp(f)$ in $J^{\rho}/J^{\rho+1}$, where $\rho=\max\{r^+,r^-\}-1$, the following equality holds
\[
\bar{\mathscr{L}}_\pp(f)=\pm\biggl(\frac{1-a_p(E)+p}{p}\biggr)\cdot{\rm Reg}^{1/2}_{\pp,{\rm der}}\cdot\sqrt{\#\Sha(E/K)}\cdot\prod_{\ell\vert N}c_\ell.
\]
\end{conj}

\subsection{A relation between the conjectures}

In this section we explain a relation between Bertolini--Darmon's  Conjecture~\ref{conj:BDconj-rank} (i.e., the ``rank part'' of Bertolini--Darmon's Conjecture~\ref{conj:BD-sqrt}) and the ``rank part'' of our Conjecture~\ref{conj:BSD-sqrt}. 

Recall that $Z_p:=E(K_\infty)\otimes\bZ_p$, and for each $n$ define the map $\Psi_n:E(K_n)\otimes\bZ_p\rightarrow Z_p[\Gamma_n]$ by
\[ 
\Psi_n(P_n)=\sum_{\sigma\in\Gamma_n} P_n^\sigma\otimes\sigma^{-1}.
\]
Letting $\pi_{n+1,n}:Z_p[\Gamma_{n+1}]\rightarrow Z_p[\Gamma_n]$ be the map induced by the projection $\Gamma_{n+1}\rightarrow\Gamma_n$,   we 
see that for all $P_{n+1}\in E(K_{n+1})\otimes\bZ_p$ we have
\[
\pi_{n+1,n}(\Psi_{n+1}(P_{n+1}))=
\sum_{\tau\in\Gamma_n}\biggl(\sum_{\substack{\sigma\in\Gamma_{n+1}\\ \sigma\vert_{K_n}=\tau}}P_{n+1}^\sigma\biggr)\otimes\tau^{-1}=
\Psi_n({\rm Norm}_{K_{n+1}/K_n}(P_{n+1})).
\]
It is also readily checked that $\Psi_n$ is $\Gamma_n$-equivariant. Thus setting 
\[
\mathcal{U}(K_\infty/K):=\varprojlim E(K_n)\otimes\bZ_p,
\]
where the limit is with respect to the norm maps ${\rm Norm}_{K_{n+1}/K_n}:E(K_{n+1})\otimes\bZ_p\rightarrow E(K_n)\otimes\bZ_p$, 
we obtain a $\Lambda$-linear map
\[
\Psi_\infty:\mathcal{U}(E/K_\infty)\rightarrow Z_p[[\Gamma_\infty]].
\]
The regularized Heegner points $\mathbf{z}_n$ in $(\ref{def:regularised})$ define an element $\mathbf{z}_\infty\in\mathcal{U}(E/K_\infty)$, and by definition 
the Heegner distribution $\theta=\theta_\infty$ in $(\ref{def:Heeg-dist})$ is given by
\begin{equation}\label{eq:Heeg-dist-Psi}
\theta_\infty=\Psi_\infty(\mathbf{z}_\infty).
\end{equation}

By  a slight abuse of notation, in the next proposition we let 
$J$ denote both the augmentation ideal of $\Lambda$ and of $\Lambda_{\hat\cO}$.

\begin{prop}\label{prop:BD-BSD}
Assume that 
\begin{enumerate}
\item $p=\pp\ppbar$ splits in $K$.
\item $E_p$ is irreducible as a $G_K$-module.
\item $\Sha(E/K_n)_{p^\infty}$ is finite  for all $n$.
\item $E_p$ is ramified for every prime $\ell\vert N$.
\item $p\nmid N\varphi(ND_K)$.
\end{enumerate}
Then we have the implication
\[
\mathscr{L}_\pp(f)\in J_{}^\rho\quad\Longrightarrow\quad\theta_\infty\in Z_p\otimes_{\bZ_p} J_{}^\rho.
\] 
\end{prop}

\begin{proof}
In light of $(\ref{eq:Heeg-dist-Psi})$ and the $\Lambda$-linearity of $\Psi_\infty$, it suffices to show the implication
\begin{equation}\label{eq:impl}
\mathscr{L}_\pp(f)\in J^\rho\quad\Longrightarrow\quad\mathbf{z}_\infty\in J^\rho\mathcal{U}(E/K_\infty).
\end{equation}
Suppose $\mathscr{L}_\pp(f)\in J^\rho$. 
By our assumption that $\#\Sha(E/K_n)_{p^\infty}<\infty$ for all $n$, we may identify $\mathcal{U}(E/K_\infty)$ with
\[
{\rm Sel}(K_\infty,T):=\varprojlim S_p(E/K_n),
\]
where the limit is with respect to the corestriction maps. Let $\gamma\in\Gamma_\infty$ be a topological generator. Using that $p\nmid\mathbf{z}_\infty$ by \cite[Thm.~B]{cornut} (taking $q=p$ in \emph{loc.cit.}) and the Weierstrass preparation theorem, we see that it suffices to solve the equation 
\[
\mathbf{z}_\infty=(\gamma-1)^\rho\cdot\mathbf{z}^{(\rho)}_\infty
\]
in $\bQ_p\otimes_{\bZ_p}{\rm Sel}(K_\infty,T)_{\hat{\cO}}$, where ${\rm Sel}(K_\infty,T)_{\hat{\cO}}$ denotes the extension of scalars to $\Lambda_{\hat\cO}$ of the $\Lambda$-module 
${\rm Sel}(K_\infty,T)$.
By \cite[Thm.~5.7]{cas-hsieh1} (see also \cite[Thm.~A.1]{cas-BF}) and \cite[Lem.~6.4]{cas-2var}, there is an injective $\Lambda_{\hat\cO}$-linear map $\mathfrak{L}_\pp:{\rm Sel}(K_\infty,T)_{\hat\cO}\rightarrow\Lambda_{\hat{\cO}}$ with finite cokernel  such that 
\begin{equation}\label{eq:ERL}
\mathfrak{L}_\pp(\mathbf{z}_\infty)=-\mathscr{L}_\pp(f)\cdot\sigma_{-1,\pp},
\end{equation}
where $\sigma_{-1,\pp}\in\Gamma_\infty$ has order two. Thus $\mathfrak{L}_\pp$ becomes an isomorphism upon tensoring with $\bQ_p$, and using the above observations the implication $(\ref{eq:impl})$ follows immediately from $(\ref{eq:ERL})$. 
\end{proof}

\section{Main results}

\subsection{Statements}

We make the following hypotheses on the triple $(E,p,K)$, where we let $\rho_{E,p}:G_{\bQ}\rightarrow{\rm Aut}_{\mathbf{F}_p}(E_p)$ be the Galois representation of the $p$-torsion of $E$.

\begin{hyp}\label{hyp:running}\hfill
\begin{enumerate}
\item $p\nmid 2N$ is a prime of good ordinary reduction for $E$. 
\item $\rho_{E,p}$ is ramified at every prime $\ell\vert N$.
\item Every prime $\ell\vert N$ splits in $K$.
\item $\rho_{E,p}$ 
is surjective.
\item $p=\pp\overline{\pp}$ splits in $K$.
\item $a_p(E)\not\equiv{1}\pmod{p}$. 
\end{enumerate}
\end{hyp}

Note that, for a given $E/\bQ$, conditions (1), (2), (4), and (6) exclude only finitely many primes $p$ by \cite{serre-open}, while conditions (3) and (5) are needed for the construction of $L_\pp(f)\in\Lambda_{\hat{\cO}}$ in \cite{cas-hsieh1}.   
Under these hypotheses, the module $\Sel_\pp(K_\infty,E_{p^\infty})$ is known to be $\Lambda$-cotorsion, 
and we let $F_\pp(f)\in\Lambda$ be a characteristic power series for its Pontryagin dual $X_\pp$. 

\begin{thm}\label{thm:A}
Assume Hypotheses~\ref{hyp:running} and that $\Sha(E/K)_{p^\infty}$ is finite. Then 
\[
{\rm ord}_JF_\pp(f)\geqslant 2(\max\{r^+,r^-\}-1),
\]
where $r^\pm={\rm rank}_\bZ E(K)^\pm$, 
and letting $\bar{F}_\pp(f)$ be the natural image of $F_\pp(f)$ in $J^{2\rho}/J^{2\rho+1}$, where $\rho=\max\{r^+,r^-\}-1$, we have
\[
\bar{F}_\pp(f)=
p^{-2}\cdot{\rm Reg}_{\pp,\rm der}\cdot\#\Sha(E/K)_{p^\infty}
\]
up to a $p$-adic unit.
\end{thm}

For comparison with the prediction of Conjecture~\ref{conj:BSD} (together with Conjecture~\ref{conj:IMC}), recall that, as noted in Remark~\ref{rem:p-adic-unit}, our hypotheses imply that the terms $1-a_p(E)+p$ and $c_\ell$ for $\ell\vert N$ are all $p$-adic units.

\begin{rem}\label{rem:JSW}
If ${\rm rank}_\bZ E(K)=1$ and $\Sha(E/K)_{p^\infty}$ is finite, then the module $\Sel_\pp(K,E_{p^\infty})$ is finite (see Lemma~\ref{lem:relation}), and therefore the image of $F_\pp(f)$ under the augmentation map 
\begin{equation}\label{eq:aug}
\epsilon:\Lambda_{\hat{\cO}}\rightarrow\hat{\cO}\nonumber
\end{equation}
is nonzero. It follows that in this case the inequality in Theorem~\ref{thm:A} is an equality, and letting $F_\pp(f)(0)\in\hat{\cO}$ denote the image of $F_\pp(f)$ under $\epsilon$, the leading coefficient formula of Theorem~\ref{thm:A} reduces to the equality (up to a $p$-adic unit) 
\[
F_\pp(f)(0)=
p^{-2}\cdot\#\Sha(E/K)_{p^\infty}\cdot\biggl(\frac{\log_{\omega_E}(y)}{[E(K):\bZ.y]}\biggr)^2,
\] 
where $y\in E(K)$ is a point of infinite order with $p^{-1}\log_{\omega_E}(y)\not\equiv 0\pmod{p}$.  
Thus Theorem~\ref{thm:A} extends the anticyclotomic control theorem in \cite[Thm.~3.3.1]{JSW} 
to arbitrary ranks.
\end{rem}



Under Hypotheses~\ref{hyp:running} (in fact, slightly weaker hypotheses suffice), and assuming that
\begin{equation}\label{eq:star}
\textrm{either $N$ is squarefree, or there are at least two primes $\ell\Vert N$},\tag{$\star$}
\end{equation}
the Iwasawa--Greenberg main conjecture for $L_\pp(f)$ 
is proved (see \cite[Cor.~7.7]{BCK-PRconj}) by building on work of Howard \cite{howard-bipartite} and W.~Zhang \cite{zhang-Kolyvagin} .  Thus Theorem~\ref{thm:A} yields the following result towards Conjecture~\ref{conj:BSD}.

\begin{cor}\label{cor:B}
Assume Hypotheses~\ref{hyp:running}, that $\Sha(E/K)_{p^\infty}$ is finite, and that {\rm (\ref{eq:star})} holds. 
Then
\[
{\rm ord}_JL_\pp(f)\geqslant 2(\max\{r^+,r^-\}-1),
\]
and letting $\bar{L}_\pp(f)$ be the natural image of $L_\pp(f)$ in $J^{2\rho}/J^{2\rho+1}$, where $\rho=\max\{r^+,r^-\}-1$, we have
\[
\bar{L}_\pp(f)=
p^{-2}\cdot
{\rm Reg}_{\pp,{\rm der}}\cdot\#\Sha(E/K)_{p^\infty}
\]
up to a $p$-adic unit. 
\end{cor}

In particular, Corollary~\ref{cor:B} shows the inclusion $\mathscr{L}_\pp(f)\in J^\rho$, where $\rho=\max\{r^+,r^-\}-1$. In light of Proposition~\ref{prop:BD-BSD}, this implies the following result, which yields one of the inequalities in the ``rank part'' of Bertolini--Darmon's Conjectures~\ref{conj:BD-general} and \ref{conj:BD-sqrt}.

\begin{cor}\label{cor:C}
Assume Hypotheses~\ref{hyp:running}, that $\Sha(E/K_n)_{p^\infty}$ is finite for all $n$, that $p\nmid N\varphi(ND_K)$, and that {\rm (\ref{eq:star})} holds. Then we have the inclusion 
\[
\theta\in Z_p\otimes J^{\rho},
\] 
where $\rho=\max\{r^+,r^-\}-1$.
\end{cor}

\subsection{Proof of Theorem~\ref{thm:A}}


Note by (5) and (6) in Hypotheses~\ref{hyp:running} we have $p\nmid\#E(\mathbf{F}_v)$ for every prime $v$ of $K$ above $p$, where $\mathbf{F}_v=\mathbf{F}_p$ is the residue field of $K$ at $v$, and by \cite[\S{4}]{mazur-18} and condition (1) this implies that the local norm maps
\begin{equation}\label{eq:norm}
{\rm Norm}_v:E(K_{n,v})
\rightarrow E(K_v)
\end{equation}
are surjective for all primes $v$ of $K$ and all finite extensions $K_n/K$ contained in $K_\infty$. (Here $E(K_{n,v})$ denotes $\bigoplus_{w\vert v}E(K_{n,w})$, where the sum is over all places $w$ of $K_n$ lying above $v$, and similar conventions for cohomology will be applied below.)

Define
\begin{align*}
{\rm H}^1_{\rm fin}(K_{n,v},E_{p^m})&:=E(K_{n,v})/p^mE(K_{n,v}),\\
{\rm H}^1_{\rm sing}(K_{n,v},E_{p^m})&:=\frac{{\rm H}^1(K_{n,v},E_{p^m})}{{\rm H}^1_{\rm fin}(K_{n,v},E_{p^m})}\simeq{\rm H}^1(K_{n,v},E)_{p^m},
\end{align*}
where the last identification follows from Tate's local duality.

\begin{defn}\label{def:adm}
As in \cite{bdIMC}, we say that a rational prime $q\nmid pN$ is \emph{$m$-admissible} for $E$ if
\begin{enumerate}
\item $q$ is inert in $K$, 
\item $q\not\equiv\pm{1}\pmod{p}$,
\item $p^m$ divides $q+1-a_q(E)$ or $q+1+a_q(E)$.
\end{enumerate}
%
We say that a finite set of rational primes $\Sigma$ is an \emph{$m$-admissible set} for $E$ if every $q\in\Sigma$ is an $m$-admissible prime for $E$ and the restriction map 
\[
\Sel_{\pp}(K,E_{p^m})\rightarrow\bigoplus_{q\in\Sigma}{\rm H}_{\rm fin}^1(K_q,E_{p^m})
\] 
is injective.
\end{defn}

\begin{rem}
As shown in \cite[Lem.~2.23]{BD-derived-DMJ} by an application of  \u{C}ebotarev's density theorem, $m$-admissible sets for $E$ always exist, and it follows from the argument in the proof given there that one can in fact always find $m$-admissible sets for $E$ with $\#\Sigma={\rm dim}_{\mathbf{F}_p}(\Sel_{\pp}(K,E_{p^m})\otimes\mathbf{F}_p)$.
\end{rem}

Following the notations introduced in \S\ref{subsec:Sel} assume that the finite set $S$ contains $\Sigma$, and let
\[
\Sel_{\pp}^\Sigma(K_n,E_{p^m}):={\rm ker}\biggl\{{\rm H}^1(\mathfrak{G}_{K_n,S},E_{p^m})\rightarrow\bigoplus_{q\in S\smallsetminus\Sigma}{\rm H}_{\rm fin}^1(K_{n,q},E_{p^m})\biggr\}
\] 
be the Selmer group $\Sel_\pp(K_n,E_{p^m})$ relaxed at the places in $\Sigma$. The next lemma underlies the usefulness of $m$-admissible sets.

\begin{lem}\label{lem:free}
Let $\Sigma$ be an $m$-admissible set for $E$. Then for every $n$ the modules 
\[
\bigoplus_{q\in\Sigma}{\rm H}^1_{\rm fin}(K_{n,q},E_{p^m}),\quad
\bigoplus_{q\in\Sigma}{\rm H}^1_{\rm sing}(K_{n,q},E_{p^m}),\quad
\Sel_{\pp}^\Sigma(K_n,E_{p^m})
\]
are free $(\bZ/p^m\bZ)[\Gamma_n]$-modules of rank $\#\Sigma$, and there is an exact sequence
\begin{equation}\label{eq:ES}
0\rightarrow\Sel_{\pp}(K_n,E_{p^m})\rightarrow\Sel_{\pp}^\Sigma(K_n,E_{p^m})\rightarrow\bigoplus_{q\in\Sigma}{\rm H}^1_{\rm sing}(K_{n,q},E_{p^m})\xrightarrow{\delta}\Sel_{\pp}(K_n,E_{p^m})^\vee\rightarrow 0,
\end{equation}
where $\delta$ is the dual to the natural restriction map.
\end{lem}

\begin{proof}
This is well-known, but we recall the arguments for the convenience of the reader. Let $q$ be an $m$-admissible prime for $E$, and denote by $\mathfrak{Q}$ the prime of $K$ lying above $q$. Then $E_{p^m}$ is unramified as $G_{K_\mathfrak{Q}}$-module, and the action of the Frobenius element at $\mathfrak{Q}$ yields a decomposition
\[
E_{p^m}\simeq(\bZ/p^m\bZ)\oplus(\bZ/p^m\bZ)(1)
\]
as ${\rm Gal}(K_{\mathfrak{Q}}^{\rm unr}/K_{\mathfrak{Q}})$-modules. From this an easy calculation shows that ${\rm H}^1_{\rm fin}(K_q,E_{p^m})$ and ${\rm H}^1_{\rm sing}(K_q,E_{p^m})$ are both free of rank one over $\bZ/p^m\bZ$ (see e.g. \cite[Lem.~2.6]{bdIMC}). Since $\mathfrak{Q}$ splits completely in $K_n/K$, the freeness claims for the first two modules follow. 

By Poitou--Tate duality, to establish the exactness of $(\ref{eq:ES})$ it suffices to establish injectivity of the restriction map
\begin{equation}\label{eq:res-sigma}
\Sel_{\pp}(K_n,E_{p^m})\rightarrow\bigoplus_{q\in\Sigma}{\rm H}^1_{\rm fin}(K_{n,q},E_{p^m})
\end{equation}
(indeed, this implies surjectivity of $\delta$).  
Arguing by contradiction, suppose 
that the kernel $\mathcal{K}$ of this map is nonzero. Then we can find a nonzero element $s\in\mathcal{K}$ 
which is fixed by $\Gamma_n$, since $\Gamma_n$ is a $p$-group. 
However, the surjectivity of the local norm maps in $(\ref{eq:norm})$ implies that the restriction map
\begin{equation}\label{eq:res-iso}
\Sel_{\pp}(K,E_{p^m})\rightarrow\Sel_{\pp}(K_n,E_{p^m})^{\Gamma_n}
\end{equation}
is an isomorphism (see \cite[Prop.~1.6]{BD-derived-AJM}), and so $s$ gives rise a nonzero element in the kernel of $\Sel_{\pp}(K,E_{p^m})\rightarrow\bigoplus_{q\in\Sigma}{\rm H}^1_{\rm fin}(K,E_{p^m})$, contradicting the $m$-admissibility of $\Sigma$. Thus the exactness of (\ref{eq:ES}) follows, and with this the freeness claims for 
the module $\Sel_{\pp}^\Sigma(K_n,E_{p^m})$ are shown by a counting argument in \cite[Thm.~3.2]{BD-derived-DMJ}.
\end{proof}

Recall that $F_\pp(f)\in\Lambda$ is a characteristic power series for the Pontryagin dual 
\[
X_\pp:=\Sel_\pp(K_{\infty},E_{p^\infty})^\vee.
\] 
Denote by $\Sel_\pp(K,E_{p^\infty})_{/{\rm div}}$ the quotient of $\Sel_\pp(K,E_{p^\infty})$ by its maximal divisible subgroup. The next result reduces the proof of Theorem~\ref{thm:A} to the calculation of $\#(\Sel_\pp(K,E_{p^\infty})_{/{\rm div}}$, which is carried out in $\S\ref{subsec:calc-sel-div}$. 

\begin{prop}\label{prop:2.23}
Assume Hypotheses~\ref{hyp:running} and that $\Sha(E/K)_{p^\infty}$ is finite. Then 
\begin{equation}\label{eq:ineq-p}
{\rm ord}_JF_\pp(f)\geqslant 2(\max\{r^+,r^-\}-1),
\end{equation}
and letting $\bar{F}_\pp(f)$ be the natural image of $F_\pp(f)$ in $J^{2\rho}/J^{2\rho+1}$, where $\rho=\max\{r^+,r^-\}-1$, we have
\[
\bar{F}_\pp(f)=\#(\Sel_\pp(K,E_{p^\infty})_{/{\rm div}})\cdot{\rm det}(A)^{-2}\cdot\prod_{k=1}^{p-1}\mathcal{R}_\pp^{(k)}
\]
up to a $p$-adic unit.
\end{prop}

The rest of the section is devoted to the proof of Proposition~\ref{prop:2.23}, for which we shall suitably adapt the arguments in \cite[\S{2.5}]{BD-derived-AJM}. 


%

Define
\begin{equation}\label{eq:Tate-equiv}
\langle\;,\;\rangle_{K_n/K,m}:\bigoplus_{q\in\Sigma}{\rm H}^1(K_{n,q},E_{p^m})\times\bigoplus_{q\in\Sigma}{\rm H}^1(K_{n,q},E_{p^m})\rightarrow(\bZ/p^m\bZ)[\Gamma_n] 
\end{equation}
by the rule 
\[
\langle x,y\rangle_n:=\sum_{\sigma\in\Gamma_n}\langle x,y^\sigma\rangle_{K_{n},m}\cdot\sigma^{-1},
\]
where $\langle\;,\;\rangle_{K_n,m}:\bigoplus_{q\in\Sigma}{\rm H}^1(K_{n,q},E_{p^m})\times\bigoplus_{q\in\Sigma}{\rm H}^1(K_{n,q},E_{p^m})\rightarrow\bZ/p^m\bZ$ is the natural extension of the local Tate pairing. 

\begin{lem}\label{lem:properties}
The pairing $\langle\;,\;\rangle_{K_n/K,m}$ is symmetric, non-degenerate, and Galois-equivariant, and the images of $\bigoplus_{q\in\Sigma}{\rm H}^1_{\rm fin}(K_{n,q},E_{p^m})$ and $\Sel_\pp^\Sigma(K_n,E_{p^m})$ are isotropic for this pairing. 
\end{lem}

\begin{proof}
All the claims except the last one follow from the corresponding properties of the local Tate pairing, while the isotropy of $\Sel_\pp^\Sigma(K_n,E_{p^m})$ follows from the global reciprocity law of class field theory.
\end{proof}

In what follows, we take $m=n$, and set 
\[
R_n:=(\Z/p^n\Z)[\Gamma_n],\quad\quad \langle\;,\;\rangle_n:=\langle\;,\;\rangle_{K_n/K,n} 
\]
for ease of notation. 

As shown in the proof of Lemma~\ref{lem:free}, the natural map $\Sel_\pp(K_n,E_{p^n})\rightarrow\bigoplus_{q\in\Sigma}{\rm H}^1(K_{n,q},E_{p^n})$ is injective, and we can write
\begin{equation}\label{eq:intersection}
\Sel_{\pp}(K_n,E_{p^n})=\biggl(\bigoplus_{q\in\Sigma}{\rm H}^1_{\rm fin}(K_{n,q},E_{p^n})\biggr)\cap\Sel_{\pp}^\Sigma(K_n,E_{p^n}),
\end{equation}
with the modules in the intersection being each free $R_n$-modules of rank $\#\Sigma$. 
By Lemma~\ref{lem:properties}, $\langle\;,\;\rangle_{n}$ restricts to a non-degenerate pairing
\[
[\;,\;]_n:\bigoplus_{q\in\Sigma}{\rm H}^1_{\rm fin}(K_{n,q},E_{p^n})\times\bigoplus_{q\in\Sigma}{\rm H}^1_{\rm sing}(K_{n,q},E_{p^n})\rightarrow R_n,
\]
and with a slight abuse of notation we define 
\[
\langle\;,\;\rangle_n:\bigoplus_{q\in\Sigma}{\rm H}^1_{\rm fin}(K_{n,q},E_{p^n})\times\Sel_\pp^\Sigma(K_n,E_{p^n})\rightarrow R_n
\]
by $\langle x, y\rangle_n:=[x,\lambda(y)]_n$, 
where $\lambda$ is the natural map $\Sel_\pp^\Sigma(K_n,E_{p^n})\rightarrow\bigoplus_{q\in\Sigma}{\rm H}^1_{\rm sing}(K_{n,q},E_{p^n})$.

\begin{lem}\label{lem:fitting}
Let $\mu_n:\Lambda\rightarrow R_n$ be the map induced by the projection $\Gamma_\infty\rightarrow\Gamma_n$. Then
\[
\mu_n(F_\pp)={\rm Fitt}_{R_n}(\Sel_\pp(K_n,E_{p^n})^\vee)
=\det(\langle x_i,y_j\rangle_n)_{1\leqslant i,j\leqslant\#\Sigma},
\]
where $x_1,\dots,x_{\#\Sigma}$ and $y_1,\dots,y_{\#\Sigma}$ are any $R_n$-bases for $\bigoplus_{q\in\Sigma}{\rm H}^1_{\rm fin}(K_{n,q},E_{p^n})$ and $\Sel_{\pp}^\Sigma(K_n,E_{p^n})$, respectively.
\end{lem}

\begin{proof}
Letting $\gamma_n\in\Gamma_n$ be a generator, the first equality follows from the natural isomorphism
\[
X_\pp/(\gamma_n-1,p^n)X_\pp\simeq\Sel_\pp(K_n,E_{p^n})^\vee
\]
together with standard properties of Fitting ideals, and the second equality follows from the fact that by Lemma~\ref{lem:free} we have a presentation
\[
R_n^{\#\Sigma}\xrightarrow{M}R_n^{\#\Sigma}\rightarrow\Sel_\pp(K_n,E_{p^n})^\vee\rightarrow 0
\]
with $M$ given by a 
matrix with entries
 $m_{i,j}=[x_i,\lambda(y_j)]_n=\langle x_i,y_j\rangle_n$ (see \cite[Lem.~2.25 and Lem.~2.26]{BD-derived-AJM} for details).
\end{proof}

Recall the filtration $\Sel_\pp(K,T)=\fSar_\pp^{(1)}\supset\fSar_\pp^{(2)}\supset\cdots\supset\fSar_\pp^{(p)}$ in $(\ref{eq:filp})$.  
Letting $\bar{\fSar}_{\pp,n}^{(k)}$ be the natural image of $\fSar_{\pp}^{(k)}$ in $\Sel_\pp(K,E_{p^n})$ we obtain a filtration
\begin{equation}\label{eq:filp-n}
\Sel_\pp(K,E_{p^n})\supset\bar{\fSar}_{\pp,n}^{(1)}\supset\bar{\fSar}_{\pp,n}^{(2)}\supset\cdots\supset\bar{\fSar}_{\pp,n}^{(p)}
\end{equation}
with $\bar{\fSar}_{\pp,n}^{(k)}/\bar{\fSar}_{\pp,n}^{(k+1)}\simeq(\bZ/p^n\bZ)^{e_k}$, for $1\leqslant k\leqslant p-1$,  and $\bar{\fSar}_{\pp,n}^{(p)}\simeq(\bZ/p^n\bZ)^{d_p}$ for $d_p={\rm rank}_{\bZ_p}\fSar_\pp^{(p)}$. 

From $(\ref{eq:intersection})$ (using that $(\ref{eq:res-iso})$ is an isomorphism), we see that 
\begin{equation}\label{eq:int-K}
\Sel_\pp(K,E_{p^n})=\biggl(\bigoplus_{q\in\Sigma}{\rm H}^1_{\rm fin}(K_{q},E_{p^n})\biggr)\cap\Sel_{\pp}^\Sigma(K,E_{p^n})
\end{equation}
with the modules in the intersection being free over $\bZ/p^n\bZ$ of rank $\#\Sigma$. 

Let $\bar{x}_1,\dots,\bar{x}_{\#\Sigma}$ and $\bar{y}_1,\dots,\bar{y}_{\#\Sigma}$ be $\bZ/p^n\bZ$-bases for $\bigoplus_{q\in\Sigma}{\rm H}^1_{\rm fin}(K_{q},E_{p^n})$ and $\Sel_{\pp}^\Sigma(K,E_{p^n})$, respectively, which are adapted to the filtration $(\ref{eq:filp-n})$, meaning that the first $r$ vectors $\bar{x}_1,\dots,\bar{x}_r$ are a basis for $\bar{\fSar}_{\pp,n}^{(1)}\subset\Sel_\pp(K,E_{p^n})$ with the images of $\bar{x}_{h_k},\dots\bar{x}_{h_k+e_k}$ in $\bar{\fSar}_{\pp,n}^{(k)}/\bar{\fSar}_{\pp,n}^{(k+1)}$ giving a basis for $\bar{\fSar}_{\pp,n}^{(k)}/\bar{\fSar}_{\pp,n}^{(k+1)}$ ($1\leqslant k\leqslant p-1)$ and $\bar{x}_{h_p},\dots,\bar{x}_{h_p+d_p}$ a basis for $\bar{\fSar}_{\pp,n}^{(p)}$, and similarly for $\bar{y}_1,\dots,\bar{y}_{\#\Sigma}$. On the other hand, let $x_1',\dots, x'_{\#\Sigma}$ and $y_1',\dots,y'_{\#\Sigma}$ be any $R_n$-bases for $\bigoplus_{q\in\Sigma}{\rm H}^1_{\rm fin}(K_{n,q},E_{p^n})$ and $\Sel_{\pp}^\Sigma(K_n,E_{p^n})$, respectively, and set 
\[
\bar{x}_i':={\rm cor}_{K_n/K}(x'_i),\quad \bar{y}_i':={\rm cor}_{K_n/K}(y'_i). 
\]
Then there exist matrices $\bar{M}$ and $\bar{N}$ in ${\rm GL}_{\#\Sigma}(\bZ/p^n\bZ)$ taking $(\bar{x}_1',\dots,\bar{x}_{\#\Sigma}')\mapsto(\bar{x}_1,\dots,\bar{x}_{\#\Sigma})$ and $(\bar{y}_1',\dots,\bar{y}_{\#\Sigma}')\mapsto(\bar{y}_1,\dots,\bar{y}_{\#\Sigma})$, respectively,  and letting $M, N\in{\rm GL}_{\#\Sigma}(R_n)$ be any lifts of $\bar{M}, \bar{N}$ under the map
${\rm GL}_{\#\Sigma}(R_n)\rightarrow{\rm GL}_{\#\Sigma}(\bZ/p^n\bZ)$ induced by the augmentation 
\[
\epsilon:R_n\rightarrow\bZ/p^n\bZ,
\]
the images of $(x_1',\dots,x'_{\#\Sigma})$, $(y_1',\dots,y'_{\#\Sigma})$ under $M$, $N$ are $R_n$-bases $(x_1,\dots,x_{\#\Sigma})$, $(y_1,\dots,y_{\#\Sigma})$ satisfying
\[
{\rm cor}_{K_n/K}(x_i)=\bar{x}_i,\quad {\rm cor}_{K_n/K}(y_i)=\bar{y}_i. 
\]

\begin{lem}\label{lem:2.27}
With the above choice of $R_n$-bases $x_1,\dots,x_{\#\Sigma}$ and $y_1,\dots,y_{\#\Sigma}$, we have
\[
\epsilon(\det(\langle x_i,y_j\rangle_n)_{r+1\leqslant i,j\leqslant\#\Sigma})=u\cdot\#(\Sel_\pp(K,E_{p^\infty})_{/{\rm div}})
\]
for some $u\in(\bZ/p^n\bZ)^\times$.
\end{lem}

\begin{proof}
Write
\[
\Sel_\pp(K,E_{p^\infty})_{/{\rm div}}\simeq\bZ/p^{s_1}\bZ\oplus\cdots
\oplus\bZ/p^{s_k}\bZ
\]
Taking $n$ from the outset to be sufficiently large, we may assume that $s_i<n$ for all $i$. Denote by $\mathfrak{X}_\pp^\Sigma(K,E)_{p^n}$ the image of $\Sel_\pp^\Sigma(K,E_{p^n})$ under the natural map 
\[
{\rm H}^1(K,E_{p^n})\rightarrow{\rm H}^1(K,E)_{p^n}.
\] 
Since the elements in $\bar{y}_1,\dots,\bar{y}_r$ are in $\Sel_\pp(K,E_{p^n})$ and $\bar{y}_1,\dots,\bar{y}_{\#\Sigma}$ is a basis for $\Sel_\pp^\Sigma(K,E_{p^n})$, we see that the natural surjection $\Sel^\Sigma_\pp(K,E_{p^n})\twoheadrightarrow\mathfrak{X}_{\pp}^\Sigma(K,E)_{p^n}$ identifies  $\mathfrak{X}_\pp^\Sigma(K,E)_{p^n}$ with the span of $\bar{y}_{r+1},\dots\bar{y}_{\#\Sigma}$  and we have an exact sequence
\[
0\rightarrow
\Sel_\pp(K,E_{p^n})\rightarrow\mathfrak{X}_\pp^\Sigma(K,E)_{p^n}\rightarrow\lambda(\Sel_\pp^\Sigma(K,E_{p^n}))\rightarrow 0.
\]
Thus we find that
\[
\lambda(\Sel_\pp^\Sigma(K,E_{p^n}))\simeq p^{s_1}(\bZ/p^n\bZ)\oplus\cdots\oplus p^{s_k}(\bZ/p^n\bZ)\oplus(\bZ/p^n\bZ)^{\#\Sigma-r-k},
\]
and choosing the basis elements $\bar{x}_{r+1},\dots,\bar{x}_{\#\Sigma}$ and $\bar{y}_{r+1},\dots,\bar{y}_{\#\Sigma}$ so that $\langle\bar{x}_i,\bar{y}_j\rangle_{K,n}=p^{s_i}\delta_{ij}$, 
the result follows using the relation
\[
\epsilon(\langle x_i,y_j\rangle_{n})=-\langle\bar{x}_i,\bar{y}_j\rangle_{K,n},
\]
which is 
immediate from the compatibility of the local Tate pairing with respect to corestriction (see \cite[Prop.~2.10]{BD-derived-DMJ}).
\end{proof}

Fix a generator $\gamma_n\in\Gamma_n$, and set
\begin{equation}\label{eq:Sel-p-n}
\fS_{\pp,n}^{(k)}:=\Sel_\pp(K,E_{p^n})\cap(\gamma_n-1)^{k-1}\Sel_\pp(K_n,E_{p^n}).\nonumber
\end{equation}
Then by definition\footnote{Since our hypotheses imply that $\Sel_\pp(K,T)$ is free; in general $\fS_\pp^{(k)}$ is defined in \cite{BD-derived-AJM} as the $p$-adic saturation of $\varprojlim_n\fS_{\pp,n}^{(k)}$ in $\Sel_\pp(K,T)$.}
\[
\fSar_\pp^{(k)}:=\varprojlim_n\fS_{\pp,n}^{(k)},
\] 
where the limit is with respect to the natural maps induced by the multiplication $E_{p^{n+1}}\rightarrow E_{p^n}$, and we have $\bar{\fSar}_{\pp,n}^{(k)}\subset \fSar_{k,n}^{(k)}$. 
Let
$\widetilde{x}_{h_k+1},\dots,\widetilde{x}_{h_k+e_k};\widetilde{y}_{h_k+1},\dots,\widetilde{y}_{h_k+e_k}\in\Sel_\pp(K_n,E_{p^n})$ be such that 
\begin{equation}\label{eq:divide}
(\gamma_n-1)^{k-1}\widetilde{x}_{h_k+i}=\bar{x}_{h_k+i},\quad (\gamma_n-1)^{k-1}\widetilde{y}_{h_k+i}=\bar{y}_{h_k+i}.
\end{equation}

For $0\leqslant k\leqslant p$, let $D_n^{(k)}\in R_n$ be the derivative operator
\[
D_n^{(k)}:=(-1)^k\gamma_n^{-k}\sum_{i=0}^{p^n-1}\binom{i}{k}\gamma_n^i
\]
(so $D_n^{(0)}=\sum_{\gamma\in\Gamma_n}\gamma$ is the norm map) introduced in \cite[\S{3.1}]{darmon-PhD}. 

\begin{claim} For every $1\leqslant k\leqslant p$, there exist elements $x'_{h_k+1},\dots,x'_{h_k+e_k}\in\bigoplus_{q\in\Sigma}{\rm H}^1_{\rm fin}(K_{n,q},E_{p^n})$ and $y'_{h_k+1},\dots,y'_{h_k+e_k}\in\Sel_\pp^\Sigma(K_n,E_{p^n})$ satisfying
\begin{equation}\label{eq:der-p}
D_n^{(k-1)}(x'_{h_k+i})=\widetilde{x}_{h_k+i},\quad 
D_n^{(k-1)}(y'_{h_k+i})=\widetilde{y}_{h_k+i}.
\end{equation}
\end{claim}

To see this, note that by $(\ref{eq:int-K})$ and the definition of $n$-admissible set we may view the $\bar{x}_{e_k+i}$ as elements in $\bigoplus_{q\in\Sigma}{\rm H}^1_{\rm fin}(K_q,E_{p^n})=(\bigoplus_{q\in\Sigma}{\rm H}^1_{\rm fin}(K_{n,q},E_{p^n}))^{\Gamma_n}$ and by  injectivity of the restriction map $(\ref{eq:res-sigma})$, 
the equality in $(\ref{eq:divide})$ may be seen as taking place in $\bigoplus_{q\in\Sigma}{\rm H}^1_{\rm fin}(K_{n,q},E_{p^n})$. Hence by \cite[Cor.~2.4]{BD-derived-AJM} applied to $\bigoplus_{q\in\Sigma}{\rm H}^1_{\rm fin}(K_{n,q},E_{p^n})$ (which is free over $R_n$ by Lemma~\ref{lem:free}), the existence of elements $x_{h_k+i}'$ satisfying (\ref{eq:der-p}) follows. The existence of elements $y_{h_k+i}'$ satisfying (\ref{eq:der-p}) is seen similarly, viewing $(\ref{eq:divide})$ as taking place in $\Sel_\pp^\Sigma(K_n,E_{p^n})$. 

By (\ref{eq:divide}), the resulting elements $x_1',\dots,x_r'$ and $y_1',\dots,y_r'$ are $R_n$-linearly independent, and setting $x_i':=x_i$ and $y_i':=y_i$ for $r+1\leqslant i\leqslant\#\Sigma$ an argument similar to that preceding Lemma~\ref{lem:2.27} shows that, after possibly transforming the bases $x_1',\dots,x_{\#\Sigma}'$ and $y_1',\dots,y_{\#\Sigma}'$ by matrices in the kernel of the map ${\rm GL}_{\#\Sigma}(R_n)\xrightarrow{}{\rm GL}_{\#\Sigma}(\bZ/p^n\bZ)$ induced by the augmentation, 
we may assume
\begin{equation}\label{eq:A-bar}
{\rm cor}_{K_n/K}(x'_{i})=\bar{x}_{i},\quad{\rm cor}_{K_n/K}(y'_{i})=\bar{y}_{i}
\end{equation}
for all $i$.

We can now conclude the proof of Proposition~\ref{prop:2.23}.

\begin{proof}[Proof of Proposition~\ref{prop:2.23}]
Let $\sigma_p:=e_1+2e_2+\cdots+(p-1)e_{p-1}+d_p$.  
(Recall that $e_k$ are given by $(\ref{eq:e_k-p})$, and $d_p:={\rm rank}_{\bZ_p}\fS_\pp^{(p)}$, which is expected to be zero.)  
To prove the inequality $(\ref{eq:ineq-p})$,  
it is enough to show the inclusion 
\begin{equation}\label{eq:incl-p}
{\rm Fitt}_{R_n}(\Sel_\pp(K_n,E_{p^n})^\vee)\in J_n^{\sigma_p}
\end{equation}
for all $n\geqslant 1$, where $J_n$ is the augmentation ideal of $R_n$. Indeed, this implies that ${\rm ord}_J F_\pp(f)\geqslant\sigma_p$, and by Remark~\ref{rem:isotropic} we have
\[
\sigma_p=\sum_{k=1}^p{\rm rank}_{\bZ_p}\fS_\pp^{(k)}\geqslant (r-1)+(\vert r^+-r^-\vert-1)=2(\max\{r^+,r^-\}-1).
\]

As noted earlier, we may choose $n$-admissible sets $\Sigma=\Sigma_n$ with $\#\Sigma$ independent of $n$, and we assume now that the preceding constructions of bases have been carried out with such $\Sigma$.

The Galois-equivariance property of $\langle\;,\;\rangle_n$ together with $(\ref{eq:der-p})$ imply that for all $1\leqslant i\leqslant e_k$ and $y\in\Sel_\pp^\Sigma(K_n,E_{p^n})$ we have
\[
D_n^{(k-1)}(\langle x'_{h_k+i},y\rangle_n)=\langle D_n^{(k-1)}(x_{h_k+i}'),y\rangle_n=0,
\] 
using Lemma~\ref{lem:properties} for the second equality. By \cite[Cor.~2.5]{BD-derived-AJM}, it follows that $\langle x'_{h_k+i},y\rangle_n\in J_n^{k}$. 
Since $(\ref{eq:A-bar})$ readily implies the equality
\begin{equation}\label{eq:conv-change}
\det(\langle x_i',y_j'\rangle_n)_{1\leqslant i,j\leqslant\#\Sigma}
=
\det(\langle x_i,y_j\rangle_n)_{1\leqslant i,j\leqslant\#\Sigma},
\end{equation}
and by Lemma~\ref{lem:2.27} we have
\begin{equation}\label{eq:cc-2}
\det(\langle x_i,y_j\rangle_n)_{1\leqslant i,j\leqslant\#\Sigma}=
{\rm Fitt}_{R_n}(\Sel_\pp(K_n,E_{p^n})^\vee),
\end{equation}
the inclusion (\ref{eq:incl-p}) follows.

Finally, to prove the expression in Proposition~\ref{prop:2.23} for the image of $F_\pp(f)$ in $J^{\rho}/J^{\rho+1}$, where $\rho=\max\{r^+,r^-\}-1$, we may assume that $\rho=\sigma_p$ (otherwise the result is trivial, both terms in the formula being equal to zero). Then by Lemma~\ref{lem:properties}, $(\ref{eq:conv-change})$, and Lemma~\ref{lem:2.27} we get 
\begin{equation}\label{eq:cc-3}
\det(\langle x_i',y_j'\rangle_n)_{1\leqslant i,j\leqslant\#\Sigma}=
\det(\langle x_i',y_j'\rangle_n)_{1\leqslant i,j\leqslant r}\cdot
u_n\cdot\#(\Sel_\pp(K,E_{p^\infty})_{/{\rm div}})\in J_n^{\rho},
\end{equation}
for some unit $u_n\in(\bZ/p^n\bZ)^\times$. Since by (\ref{eq:divide}), (\ref{eq:der-p}), and the definition of the derived pairing $h_{\pp,n}^{(k)}$ (see \cite[p.~1526]{BD-derived-AJM}) we have
\begin{equation}\label{eq:cc-4}
h_{\pp,n}^{(k)}(\bar{x}_i,\bar{y}_j)=\langle x'_i,y'_j\rangle_n\in J_n^{k}/J_n^{k+1}
\end{equation}
for all $h_k+1\leqslant i,j\leqslant h_k+e_k$, 
combining Lemma~\ref{lem:fitting} with $(\ref{eq:conv-change})$, (\ref{eq:cc-2}), (\ref{eq:cc-3}), and (\ref{eq:cc-4}) we 
arrive at the equality
\[
\mu_n(F_\pp)=u_n\cdot 
\prod_{k=1}^{p-1}\det(h_{\pp,n}^{(k)}(\bar{x}_{i},\bar{y}_j))_{h_k+1\leqslant i,j\leqslant h_k+e_k}
\]
in $J_n^\rho/J_n^{\rho+1}$, and letting $n\to\infty$ the result follows.
\end{proof}

\subsection{Calculation of $\#(\Sel_\pp(K,E_{p^\infty})_{/{\rm div}})$}\label{subsec:calc-sel-div}

Define the $\bar\pp$-relaxed Tate--Shafarevich group 
by
\[
\Sha_{}^{\{\overline\pp\}}(E/K):=
{\rm ker}\biggl\{{\rm H}^1(K,E)\rightarrow\prod_{w\neq\bar\pp}{\rm H}^1(K_w,E)\biggr\},
\]
and let $\Sha_{}^{\{\overline\pp\}}(E/K)_{p^\infty}$ denote its $p$-primary component. Recall that by hypothesis (\ref{eq:Heeg}) the root number of $E/K$ is $-1$, so by the $p$-parity conjecture if $\Sha(E/K)_{p^\infty}$ is finite then $E(K)$ has positive rank.


\begin{lem}\label{lem:6.10}
Assume that 
$\Sha(E/K)_{p^\infty}$ is finite. Then $\Sha^{\{\ppbar\}}(E/K)_{p^\infty}$ is also finite, and we have
\[
\#\Sha_{}^{\{\ppbar\}}(E/K)_{p^\infty}=
\#\Sha(E/K)_{p^\infty}\cdot\#{{\rm coker}({\rm loc}_\pp)},
\]
where ${\rm loc}_\pp:S_p(E/K)\rightarrow E(K_\pp)\otimes\bZ_p$ is the restriction map.
\end{lem}

\begin{proof}
Define $B_\infty$ by the exactness of the sequence
\begin{equation}\label{eq:taut}
0\rightarrow\Sha(E/K)_{p^\infty}\rightarrow{\rm H}^1(K,E)_{p^\infty}\rightarrow\prod_{w}{\rm H}^1(K_w,E)_{p^\infty}\rightarrow B_\infty\rightarrow 0.\nonumber
\end{equation}
Then we have an induced exact sequence
\begin{equation}\label{eq:snake}
0\rightarrow\Sha(E/K)_{p^\infty}\rightarrow\Sha_{}^{\{\ppbar\}}(E/K)_{p^\infty}\rightarrow{\rm H}^1(K_{\ppbar},E)_{p^\infty}\xrightarrow{h_\infty} B_\infty.
\end{equation}
By surjectivity of the top right arrow in the commutative 
\[
\xymatrix{
0\ar[r] &E(K)\otimes\bQ_p/\bZ_p\ar[r]\ar[d]&{\rm H}^1(K,E_{p^\infty})\ar[r]\ar[d]\ar[rd]^-{\partial}&{\rm H}^1(K,E)_{p^\infty}\ar[d]\ar[r]&0\\
0\ar[r] &\prod_w E(K_w)\otimes\bQ_p/\bZ_p\ar[r] &\prod_w{\rm H}^1(K_w,E_{p^\infty})\ar[r] &\prod_w{\rm H}^1(K_w,E)_{p^\infty}\ar[r]&0,
}
\]
we see 
that ${\rm ker}(h_\infty)$ is the same as the kernel of the map $\delta$ in the Cassels dual exact sequence
\[
0\rightarrow{\rm Sel}_{p^\infty}(E/K)\rightarrow{\rm Sel}^{\{\ppbar\}}_{p^\infty}(E/K)\rightarrow{\rm H}^1(K_{\ppbar},E)_{p^\infty}\xrightarrow{\delta}S_p(E/K)^\vee,
\]
where ${\rm Sel}^{\{\ppbar\}}_{p^\infty}(E/K)$ is the kernel of the map ${\rm H}^1(K,E_{p^\infty})\rightarrow\prod_{w\nmid\ppbar}{\rm H}^1(K_w,E)_{p^\infty}$.
 
Using that $E_{p^\infty}^*:={\rm Hom}_{\rm cts}(E_{p^\infty},\mu_{p^\infty})\simeq T^\tau$ (which exchanges the restriction maps at $\pp$ and $\ppbar$) it follows that the kernel of $h_\infty$ is dual to the cokernel of the  map ${\rm loc}_{\pp}:S_p(E/K)\rightarrow E(K_{\pp})\otimes\bZ_p$, which is finite under our hypotheses. The result follows.
\end{proof}

The following result is an analogue of \cite[Prop.~3.2.1]{JSW} in arbitrary (co)rank.

\begin{prop}\label{prop:6.11}
Assume that 
$\Sha(E/K)_{p^\infty}$ is finite and $a_p(E)\not\equiv 1\pmod{p}$. Then
\[
\#(\Sel_\pp(K,E_{p^\infty})_{/{\rm div}})=\#\Sha(E/K)_{p^\infty}\cdot(\#{\rm coker}({\rm loc}_\pp))^2,
\]
where ${\rm loc}_\pp:S_p(E/K)\rightarrow E(K_\pp)\otimes\bZ_p$ is the restriction map.
\end{prop}

\begin{proof}
Let $y_1,\dots,y_{r-1}$ be a $\bZ_p$-basis for the kernel
\[
E_{1,\pp}(K):={\rm ker}\biggl\{E(K)\otimes\bZ_p\xrightarrow{{\rm loc}_\pp} E(K_\pp)\otimes\bZ_p\biggr\},
\]
and extend it to a $\bZ_p$-basis $y_1,\dots,y_{r-1},y_\pp$ for $E(K)\otimes\bZ_p$, so
\begin{equation}\label{eq:directsum}
E(K)\otimes\bZ_p=E_{1,\pp}(K)\oplus\bZ_p.y_\pp.
\end{equation}
Then the finite module $U$ defined by the exactness of the sequence
\begin{equation}\label{eq:U}
0\rightarrow\bZ_p.y_\pp\rightarrow E(K_\pp)\otimes\bZ_p\rightarrow U\rightarrow 0
\end{equation}
satisfies
\begin{equation}\label{eq:U}
\#U=[E(K_\pp)\otimes\bZ_p\colon{\rm loc}_\pp(E(K)\otimes\bZ_p)]=
\#{\rm coker}({\rm loc}_\pp),
\end{equation}
using the finiteness assumption on $\Sha(E/K)$ for the second equality. The hypothesis that $a_p(E)\not\equiv 1\pmod{p}$ implies that $E(K_\pp)$ has no $p$-torsion, and so $E(K_\pp)\otimes\bZ_p$ is a free $\bZ_p$-module of rank one. Tensoring (\ref{eq:U}) with $\bQ_p/\bZ_p$ therefore yields
\[
0\rightarrow V\rightarrow(\bQ_p/\bZ_p).y_\pp\rightarrow E(K_\pp)\otimes\bQ_p/\bZ_p\rightarrow 0
\]
with $\#V=
\#U$,  
and from $(\ref{eq:directsum})$ we deduce that
\begin{equation}\label{eq:ker-lam}
{\rm ker}\biggl\{E(K)\otimes\bQ_p/\bZ_p\xrightarrow{\lambda_\pp}E(K_\pp)\otimes\bQ_p/\bZ_p\biggr\}=(E_{1,\pp}(K)\otimes\bQ_p/\bZ_p)\oplus V.
\end{equation}

Now consider the $p$-relaxed Tate--Shafarevich group $\Sha^{\{p\}}(E/K)$ defined by
\[
\Sha^{\{p\}}(E/K):={\rm ker}\biggl\{{\rm H}^1(K,E)\rightarrow\prod_{w\nmid p}{\rm H}^1(K_w,E)\biggr\}.
\]
It is immediately seen that its $p$-primary part fits into the exact sequence
\[
0\rightarrow E(K)\otimes\bQ_p/\bZ_p\rightarrow {\rm Sel}_{p^\infty}^{\{p\}}(E/K)\rightarrow\Sha^{\{p\}}(E/K)_{p^\infty}\rightarrow 0,
\]
where ${\rm Sel}_{p^\infty}^{\{p\}}(E/K)$ is the kernel of the map ${\rm H}^1(K,E_{p^\infty})\rightarrow\prod_{w\nmid p}{\rm H}^1(K_w,E)_{p^\infty}$. 
Consider also the commutative diagram
\[
\xymatrix{
0 \ar[r] &E(K)\otimes\bQ_p/\bZ_p \ar[r]\ar[d]^-{\lambda_{\pp}} &{\rm Sel}_{p^\infty}^{\{p\}}(E/K)\ar[r]\ar[d] &\Sha^{\{p\}}(E/K)_{p^\infty}\ar[d]\ar[r] &0\\
0 \ar[r] &E(K_{\pp})\otimes\bQ_p/\bZ_p\ar[r] &{\rm H}^1(K_{\pp},E_{p^\infty})\ar[r] &{\rm H}^1(K_{\pp},E)_{p^\infty}\ar[r] &0
}
\]
in which the unlabeled vertical maps are given by restriction. Since the map $\lambda_{\pp}$ is surjective by our assumptions, the snake lemma applied to this diagram yields the exact sequence
\begin{equation}\label{eq:fund}
0\rightarrow{\rm ker}(\lambda_{\pp})\rightarrow\Sel_\pp(K,E_{p^\infty})\rightarrow\Sha_{}^{\{\ppbar\}}(E/K)_{p^\infty}\rightarrow 0,
\end{equation}
and hence from  $(\ref{eq:fund})$, $(\ref{eq:ker-lam})$ and $(\ref{eq:U})$ we conclude that
\begin{align*}
\#(\Sel_\pp(K,E_{p^\infty})_{/{\rm div}})=\#\Sha_{}^{\{\ppbar\}}(E/K)_{p^\infty}\cdot\#V&=\#\Sha_{}^{\{\ppbar\}}(E/K)_{p^\infty}\cdot\#{\rm coker}({\rm loc}_\pp)\\
&=\#\Sha(E/K)_{p^\infty}\cdot\#({\rm coker}({\rm loc}_\pp))^2,
\end{align*}
using Lemma~\ref{lem:6.10} for the last equality.
\end{proof}

As in the proof of Proposition~\ref{prop:6.11}, 
let $y_1,\dots, y_{r-1}$ be a $\bZ_p$-basis for the kernel ${\rm Sel}_{\rm str}(K,T)$ of
\[
{\rm loc}_\pp:E(K)\otimes\bZ_p\rightarrow E(K_\pp)\otimes\bZ_p,
\] 
and extend it to a $\bZ_p$-basis $y_1,\dots,y_{r-1},y_\pp$ for $E(K)\otimes\bZ_p$. We denote by ${\rm log}_{\omega_E}:E(K)\otimes\bZ_p\rightarrow\bZ_p$ the composition of ${\rm loc}_\pp$ with the formal group logarithm associated with a N\'{e}ron differential $\omega_E\in\Omega^1(E/\bZ_{(p)})$.

\begin{prop}\label{prop:coker}
Assume that 
$\Sha(E/K)_{p^\infty}$ is finite and $a_p(E)\not\equiv 1\pmod{p}$. 
Then 
\[
\#{\rm coker}({\rm loc}_\pp)=p^{-1}
\#(\bZ_p/\log_{\omega_E}(y_\pp)).
\]
\end{prop}

\begin{proof}
Let $E_1(K_\pp)$ be the kernel of reduction modulo $\pp$, so there is an exact sequence
\[
0\rightarrow E_1(K_\pp)\rightarrow E(K_\pp)\rightarrow E(\mathbf{F}_p)\rightarrow 0.
\]
Set $Y:=\bZ_p.y_\pp$, $Y_{\pp,1}:={\rm loc}_\pp(Y)\cap(E_1(K_\pp)\otimes\bZ_p)$ and $Z:=Y/Y_{\pp,1}$ (a finite group), and consider the commutative diagram
\[
\xymatrix{
0\ar[r] &Y_{\pp,1}\ar[r]\ar[d]^-{\lambda_{\pp,1}} &Y \ar[r]\ar[d]^-{{\rm loc}_\pp\vert_{Y}} &Z\ar[r]\ar[d]^-{}&0\\
0\ar[r] &E_1(K_\pp)\otimes\bZ_p \ar[r] &E(K_\pp)\otimes\bZ_p 
\ar[r] &E(\mathbf{F}_p)\otimes\bZ_p \ar[r] & 0.
}
\]
Since the middle vertical is injective by our choice of $y_\pp$ and  $E(\mathbf{F}_p)\otimes\bZ_p\simeq\bZ_p/(1-a_p(E)+p)$ is trivial by our assumption on $a_p(E)$, applying the snake lemma we deduce
\begin{equation}\label{eq:snake-2}
\#{\rm coker}({\rm loc}_\pp\vert_Y)\cdot\#Z=\#{\rm coker}(\lambda_{\pp,1}).
\end{equation}

On the other hand, noting that $\#Z\cdot y_\pp$ is a generator of $Y_{\pp,1}$ and the formal group logarithm induces an isomorphism ${\rm log}_{\omega_E}:E_1(K_\pp)\otimes\bZ_p\simeq p\bZ_p$ we find 
\begin{equation}\label{eq:log}
\#{\rm coker}(\lambda_{\pp,1})=\frac{\#\bZ_p/{\rm log}_{\omega_E}(\#Z\cdot y_\pp)}{\#\bZ_p/{\rm log}_{\omega_E}(E_1(K_\pp)\otimes\bZ_p)}=\#Z\cdot p^{-1}\#(\bZ_p/{\rm loc}_{\omega_E}(y_\pp)).
\end{equation}
Since clearly $\#{\rm coker}({\rm loc}_{\pp}\vert_Y)=[E(K_\pp)\otimes\bZ_p\colon{\rm loc}_\pp(S_p(E/K))]$ by the definition of $y_\pp$,  combining (\ref{eq:snake-2}) and (\ref{eq:log}) the result follows.
\end{proof}

We can now conclude the proof of Theorem~\ref{thm:A}.

\begin{proof}[Proof of Theorem~\ref{thm:A}]
By Proposition~\ref{prop:2.23} we have ${\rm ord}_{J}F_\pp(f)\geqslant 2\rho$ with $\rho=\max\{r^+,r^-\}-1$, and the equality
\begin{equation}\label{eq:2.23}
\bar{F}_\pp(f)=\#(\Sel_\pp(K,E_{p^\infty})_{/{\rm div}})\cdot{\rm det}(A)^{-2}\cdot
\prod_{k=1}^{p-1}\mathcal{R}_\pp^{(k)}
\end{equation}
in $(J^{2\rho}/J^{2\rho+1})\otimes\bQ$ up to a $p$-adic unit. On the other hand, combining Propositions~\ref{prop:6.11} and \ref{prop:coker} we obtain
\begin{equation}\label{eq:Sel-div}
\begin{aligned}
\#(\Sel_\pp(K,E_{p^\infty})_{/{\rm div}})&=\#\Sha(E/K)_{p^\infty}\cdot(\#{\rm coker}({\rm loc}_\pp))^2\\
&=\#\Sha(E/K)_{p^\infty}\cdot p^{-2}\cdot{\rm log}_{\omega_E}(y_\pp)^2,
\end{aligned}
\end{equation}
with the last equality holding up to a $p$-adic unit. Recalling the Definition~\ref{def:regp} of ${\rm Reg}_{\pp,{\rm der}}$, the proof of Theorem~\ref{thm:A} now follows from (\ref{eq:2.23}) and (\ref{eq:Sel-div}).
\end{proof}

{\bf Acknowledgements.} It is a pleasure to thank Henri Darmon and Chris Skinner for their comments on an earlier draft of this paper.
\bibliographystyle{amsalpha}
\bibliography{Rubin-refs}

\end{document}